\documentclass[11pt]{amsart}
\usepackage{amssymb}
\usepackage{amsmath}
\usepackage{amsthm}
\usepackage{amsbsy}
\usepackage{psfrag}
\usepackage{pstricks}
\usepackage{graphics}
\usepackage{graphicx}
\usepackage{ragged2e}
\usepackage{amsmath}
\usepackage{amsmath,amssymb,enumitem,ulem,multicol}
\usepackage{mathabx}

\usepackage{amssymb,tikz}
\usepackage{stackengine}
\usepackage{amsthm}
\usepackage{upgreek} 
\usepackage{bigints}
\usepackage{setspace}
\usepackage{mathrsfs}
\usepackage{mathtools}
\usepackage{xcolor}
\usepackage{xcolor}
\usepackage{float}
\usepackage[utf8]{inputenc}
\DeclareUnicodeCharacter{2212}{-}
\usepackage[english]{babel}
\usepackage{graphicx}
\usepackage{subcaption}
\theoremstyle{remark}

\def\R{\mathbb{R}}

\def\cN{\mathcal{N}}
\def\cA{\mathcal{A}}
\def\cK{\mathcal{K}}
\def\cT{\mathcal{T}}
\def\cE{\mathcal{E}}

\def\p{\partial}

\def\[{\partial}
\def\O{\Omega}
\def\ssT{{\scriptscriptstyle T}}
\def\HT{{H^2(\O,\cT_h)}}
\def\mean#1{\left\{\hskip -5pt\left\{#1\right\}\hskip -5pt\right\}}
\def\jump#1{\left[\hskip -3.5pt\left[#1\right]\hskip -3.5pt\right]}
\def\smean#1{\{\hskip -3pt\{#1\}\hskip -3pt\}}
\def\sjump#1{[\hskip -1.5pt[#1]\hskip -1.5pt]}
\def\jumptwo{\jump{\frac{\p^2 u_h}{\p n^2}}}
\def\b#1{\boldsymbol{#1}}
\def\norm #1{{\left\vert\kern-0.25ex\left\vert\kern-0.25ex\left\vert #1 
    \right\vert\kern-0.25ex\right\vert\kern-0.25ex\right\vert}}
\usepackage{ragged2e}
\usepackage{amssymb}
\usepackage{amsmath}
\usepackage{amsthm}
\usepackage{amsbsy}
\usepackage{psfrag}
\usepackage{pstricks}
\usepackage{graphics}
\usepackage{graphicx}
\usepackage{tgtermes}
\fontfamily{qcr}
\setlist[itemize]{wide=0pt,leftmargin=15pt,align=left}
\theoremstyle{plain}
\newtheorem{theorem}{Theorem}[section]
\newtheorem{lemma}[theorem]{Lemma}

\theoremstyle{remark}
\newtheorem{remark}[theorem]{Remark}

\begin{document}
\allowdisplaybreaks[4]
\numberwithin{figure}{section}
\numberwithin{table}{section}
 \numberwithin{equation}{section}
%
\title[  Quadratic Discontinuous Galerkin methods for Unilateral Contact Problem]
 {Quadratic discontinuous Galerkin finite element methods for the unilateral contact problem}
\author{Kamana Porwal}\thanks{The second author's work is supported  by CSIR Extramural Research Grant}
\address{Department of Mathematics, Indian Institute of Technology Delhi - 110016}
\email{kamana@maths.iitd.ac.in}
\author{Tanvi Wadhawan}\thanks{}
\address{Department of Mathematics, Indian Institute of Technology Delhi - 110016}
\email{maz188452@iitd.ac.in}
\date{}

\begin{abstract}
In this article,  we employ discontinuous Galerkin (DG) methods for the finite element approximation of the  frictionless unilateral contact problem using quadratic finite elements over simplicial triangulation.  We first establish an optimal \textit{a priori} error estimates under the appropriate regularity assumption on the exact solution $\b{u}$.  Further,  we analyze \textit{a posteriori} error estimates in the DG norm wherein,  the reliability and efficiency of the proposed \textit{a posteriori} 
error estimator is addressed.  The suitable construction of discrete Lagrange multiplier $\b{\lambda_h}$  and some intermediate operators  play a key role in developing  \textit{a posteriori} error analysis.  Numerical results presented on uniform and adaptive meshes illustrate and confirm the theoretical findings.  
\end{abstract}


%
%
\maketitle
\allowdisplaybreaks
\def\R{\mathbb{R}}
\def\cA{\mathcal{A}}
\def\cK{\mathcal{K}}
\def\cN{\mathcal{N}}
\def\p{\partial}
\def\O{\Omega}
\def\bbP{\mathbb{P}}
\def\cV{\mathcal{V}}
\def\cM{\mathcal{M}}
\def\cT{\mathcal{T}}
\def\cE{\mathcal{E}}
\def\bF{\mathbb{F}}
\def\bC{\mathbb{C}}
\def\bN{\mathbb{N}}
\def\ssT{{\scriptscriptstyle T}}
\def\HT{{H^2(\O,\cT_h)}}
\def\mean#1{\left\{\hskip -5pt\left\{#1\right\}\hskip -5pt\right\}}
\def\jump#1{\left[\hskip -3.5pt\left[#1\right]\hskip -3.5pt\right]}
\def\smean#1{\{\hskip -3pt\{#1\}\hskip -3pt\}}
\def\sjump#1{[\hskip -1.5pt[#1]\hskip -1.5pt]}
\def\jumptwo{\jump{\frac{\p^2 u_h}{\p n^2}}}
\section{Introduction} \label{sec:intro}
\noindent
The numerical analysis of contact problems arising in the various physical phenomena plays a substantial role in understanding processes and effects in natural sciences. From mathematical point of view, contact problems are modeled as  variational inequalities which play an essential role in solving class of various non-linear boundary value problems arising in the physical fields.  Much of the mathematical basis of modeling and detailed understanding of contact problems within the framework of variational inequality can be found in the book by  Kikuchi and Oden \cite{KO:1988:CPBook}.  Among the contact problems,  the unilateral contact represented by the frictionless Signorini model simulates the contact between a linear elastic body and a rigid foundation.  The Signorini problem is basically studied as a prototype for elliptic variational inequality (EVI) of the first kind.   In the underlying variational inequality, the non-linearity condition in the weak formulation is incorporated in a  closed and convex set on which the formulation is posed. 
This contact problem was formulated by Signorini (see \cite{Signorini}) followed by that Fichera in \cite{Fichera} carried out an extensive study of Signorini problem in the context of EVI's.  Subsequently, many researchers have done plenty of work in modeling the contact problems. 
The study of \textit{a priori} error analysis for unilateral contact problem using conforming linear finite elements can be found in \cite{AH:2009:VI, Glowinski:2008:VI,KO:1988:CPBook}.  Resorting to the higher order finite elements provide more accurate computed solution \cite{Ciarlet:1978:FEM}, compared to linear elements.  Contrarily,  less literature is available for the same.  The article \cite{apriori:2001:quad} exploits the two non-conforming quadratic approximation to the Signorini problem and derive \textit{a priori} error estimates.  In order to quantify the discretization errors, \textit{a posteriori} error estimators act as an indispensable tool. In article \cite{Hild:2005:CPE},  the authors constructed a positivity preserving interpolation operator to derive a reliable and efficient residual based \textit{a posteriori} error estimators for Signorini problem using linear finite elements.  Further,  in the article \cite{Weiss:2009},  the \textit{a posteriori} error analysis is carried out without using the positivity preserving interpolation operator. A residual based a posteriori error estimator of conforming linear finite element method for the Signorini problem is developed in \cite{Krause:2015:apost_Sig} using a suitable construction of the quasi discrete contact force density. There is hardly any work available in the literature towards the a posteriori error analysis of quadratic finite element methods for the Signorini problem.\\
\par \noindent
Unlike the standard finite element method, the discontinuous Galerkin methods introduced by Reed and Hill \cite{RH:1973:DG} work with the space of trial functions which are only piecewise continuous. These methods are well known for their flexibility for hp adaptivity.  The discontinuous property permits the usage of general meshes with hanging nodes.  In addition to this,  its property to incorporate the non-homogeneous boundary condition in the weak formulation greatly increases the accuracy and robustness of any boundary condition implementation.
Consequently,  DG methods have been applied to study several linear and non-linear partial differential equations.  We refer to \cite{ Cockburn:2000:DGM,  HW:2007:Book,  PAE:2012:DGbook,Riviere:2008:DGBook} and the references therein for the comprehensive study of DG methods.
\par \noindent
 In the past decades,  DG methods are extensively used to solve variational inequalities.  In the article \cite{WHC:2010:DGVI},  the ideas discussed in \cite{ABCM:2002:UnifiedDG} were extended to deduce an optimal \textit{a priori} error estimates of DG methods for obstacle and simplified Signorini problem.  Followed by that in article \cite{WHC:2011:DGSP} and \cite{WHC:2014:DGQSP},  DG formulations for Signorini problem and quasi static contact problem are derived using linear elements and optimal \textit{a priori} error estimates are established for these methods.  The articles  \cite{TG:2014:VIDG1, TG:2014:VIDG, WHE:2015:ApostDG, GGP:2021:DGVI} addressed  \textit{a posteriori} error control for obstacle problem using DG methods.  The study of \textit{a posteriori} error estimates on adaptive mesh for Signorini problem using discontinuous linear elements was dealt in \cite{TG:2016:VIDG1}.  Therein,  the authors carried out the analysis by defining the continuous Lagrange multiplier as a functional on $\b{H^{\frac{1}{2}}}(\Gamma_C)$.  In comparison to \cite{TG:2016:VIDG1},  M.  Walloth in the article \cite{Walloth:2019:Dg} discussed the  \textit{a posteriori} error analysis of DG methods for Signorini problem which relies on the construction of quasi discrete contact force density.  In this article, we provide a rigorous analysis for the DG discretization with quadratic polynomials to Signorini  problem, therein both \textit{a priori} and \textit{a posteriori} error analysis are discussed.  We followed the approach shown in the article \cite{TG:2016:VIDG1} and defined continuous Lagrange multiplier on $\b{H^{-\frac{1}{2}}}(\Gamma_C)$.  The interpolation operator $\b{\pi_h}$ and $\b{\beta_h}$ (defined in equation \eqref{pi_h} and equation \eqref{32},  respectively) is suitably constructed to exhibit the desired properties which are  later used to define discrete counter part of Lagrange multiplier.  \\
\par \noindent
The outline of the article is as follows: In Section 1,  we state the classical and variational formulation of the Signorini problem. In addition to this, we define an auxiliary functional on the space $\b{H^{\frac{1}{2}}}(\Gamma_C)$ pertaining to the exact solution and complimentarity conditions on the contact region.  We introduce some prerequisite notations and preliminary results in Section 2.  Therein,  we also define the discrete counterpart of the continuous problem on closed,  convex non-empty subset of a quadratic finite element space,  couple of interpolation operators and discuss several DG formulations.  Section 3 is dedicated to derive an optimal \textit{a priori} error estimates with respect to the regularity of the exact solution, followed by that in Section 4 we introduce the discrete counterpart $\b{\lambda_h}$ of Lagrange multiplier $\b{\lambda}$ on a suitable space.  Next,  we derive the sign properties of $\b{\lambda_h}$ which play a key role in the subsequent analysis.  In order to deal with the discontinuous finite elements we construct an enriching map which connects DG functions with conforming finite elements and preserves constraints on discrete functions.  Section 5 is devoted to unified \textit{a posteriori} error analysis for the various DG methods wherein,  we discuss the reliablity and efficiency of \textit{a posteriori} error estimator.  Finally in Section 6,  the theoretical results are corroborated by two numerical experiments addressing contact of a linear elastic body with the rigid obstacle.  Therein,  the optimal convergence of the error on uniform mesh and  the convergence behavior of error estimator over adaptive mesh are demonstrated for two DG methods namely, SIPG and NIPG.

\par
\section{\textbf{Signorini Contact Problem}}
\noindent
Let $\O$ represents a bounded,  polygonal linear elastic body in $\mathbb{R}^2$ with Lipschitz  boundary $\partial \Omega=\Gamma$ which is partitioned into three mutually disjoint,  relatively open sets;  the Dirichlet boundary $\Gamma_D$,  the Neumann boundary $\Gamma_N$ and the potential contact boundary $\Gamma_C$ with $meas(\Gamma_D) > 0$ and $\overline \Gamma_C \subset \Gamma~ \backslash~\overline \Gamma_D$.  
\par
\noindent
 Let $\b{u}:\Omega \rightarrow \mathbb{R}^2$ denote the displacement vector. Then,  the linearized strain tensor 
\begin{align}\label{0.1}
\b{\varepsilon}(\b{u}) := \dfrac{1}{2}\big(\b{\nabla}\b{u} +(\b{ \nabla}\b{u})^{\b{T}}\big),
\end{align}
and stress tensor
\begin{align}\label{0.2}
\b{\sigma}(\b{u}) := \mathcal{C}\b{\varepsilon}(\b{u}),
\end{align}
belong to $\b{\mathbb{S}}^2$ which is  the space of second order real symmetric tensors on $\mathbb{R}^2$ with the inner product $\b{\psi}:\b{\chi} = \psi_{ij}\chi_{ij}$ and the norm $|\b{\psi}|= (\b{\psi}:\b{\psi})^{\frac{1}{2}}~~\forall~\b{\psi},\b{\chi} \in \b{\mathbb{S}}^2$ .   In the relation \eqref{0.2},  $\mathcal{C}$ denotes the bounded,  symmetric and positive definite fourth order elasticity tensor defined on $\O \times \mathbb{S}^2 \longrightarrow \mathbb{S}^2$.  In particular,  for a homogeneous and isotropic elastic body,   the stress  tensor obeys  Hooke's law and is explicitly given by
\begin{align*}
\mathcal{C}\b{\varepsilon}(\b{u}):= 2\mu\b{\varepsilon}(\b{u})+\lambda(\text{tr} \b{\varepsilon}(\b{u}))\textbf{Id},
\end{align*}
where,  $\mu>0$ and $\lambda>0$ denote the Lam$\acute{e}$'s parameter and $ \textbf{Id}$ is an identity matrix of order 2.  
\vspace{0.3 cm}
\par 
\noindent
Throughout this article,  the vector valued functions are identified with the bold symbols whereas the scalar valued functions are written in usual way.  Let $\b{n}$ denotes the outward unit normal vector to $\Gamma$,  then for any vector $\b{v}$,  we decompose  its normal and tangential component as $v_n:= \b{v} \cdot \b{n}$ and $\b{v_{\tau}} := \b{v}- v_n\b{n}$, respectively.  Analogously,  for any tensor valued function $\b{\Psi}$ in $\b{\mathbb{S}}^2$,  we denote  $\Psi_n:= \b{\Psi}\b{n}\cdot\b{n}$ and  $\b{\Psi_{\tau}}:= \b{\Psi}\b{n}-\Psi_n\b{n}$ as its normal and tangential component,  respectively.  Further,  we will require the following decomposition formula
\begin{align}\label{decomp}
(\b{\Psi}\b{n})\cdot \b{v} = \Psi_nv_n +\b{\Psi_{\tau}}\cdot \b{v}_{\tau}.
\end{align}
In the following static contact model problem,  the linear elastic body lying on a rigid foundation is subjected to volume force of density $\b{f}$ in $\Omega$,  surface traction force $\b{g}$ on Neumann boundary together with unilateral contact condition on $\Gamma_C$.  The pointwise formulation of the contact problem is described below in subsection (\ref{strong}). 
\subsection{Strong formulation of Signorini problem:}\label{strong}
 Find the displacement vector $\b{u}: \Omega \rightarrow \mathbb{R}^2$ satisfying the equations \eqref{1.4}-\eqref{1.7},
\begin{align}
    \b{-div} ~~\b{\sigma}(\b{u}) &= \b{f} ~~~~\textit{in}~\Omega,\label{1.4}\\
    \b{u} &= \b{0} ~~~~\textit{on}~\Gamma_D, \label{1.5}\\
    \b{\sigma}(\b{u})\b{n} &= \b{g} ~~~~\textit{on}~\Gamma_N,\label{1.6}
    \end{align}
    \begin{align} \label{1.7}
       \begin{split}
    \left.\begin{aligned}
    u_n&\leq 0,~~ \sigma_n(\b{u})\leq 0 \\
    u_{n} \sigma_n(\b{u})&=0,~~ \b{\sigma_\tau}(\b{u}) =0 
    \end{aligned}\right\}  on~~\Gamma_C,
    \end{split}
    \end{align}
 where $\b{f} \in [L^2(\O)]^2$ and $\b{g} \in [L^2(\Gamma_N)]^2$.  The equilibrium equation in which volume force of density $\b{f}$ acts in $\O$ is described in \eqref{1.4}.  The equation $\eqref{1.5}$ depicts that the displacement field vanishes on Dirichlet boundary,  thus the body is clamped on $\Gamma_D$.  On the contact boundary $\Gamma_C$,  the relation \eqref{1.7} constitutes the non-penetration condition that evokes the contact stresses in the direction of the constraints.  In addition,  the contact of elastic body and deformable foundation is assumed to be frictionless,  henceforth the frictional stresses $\b{\sigma_\tau}(\b{u})$ are assumed to be zero (see Figure \ref{FIg1}).
\begin{figure}[ht!] 
	\begin{center}
		\includegraphics[height=5cm,width=9cm]{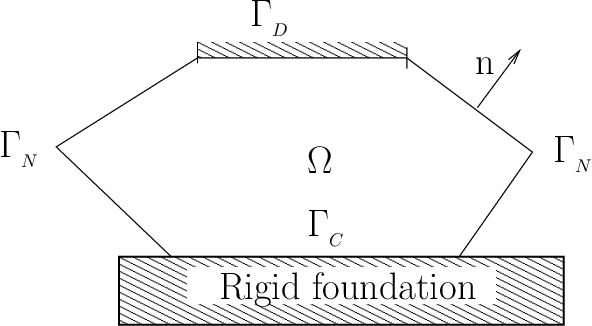}
		\caption{Physical setting of Signorini problem}
		\label{FIg1}
	\end{center}
\end{figure}
 \par
 \noindent
In the subsequent analysis,  we will make a constant use of Sobolev spaces $H^m(\O)$ which is endowed with the following norm $\|{v} \|_{H^m(\O)} := \bigg(\underset{0 \leq |\alpha| \leq m}{\sum} \|\partial^{\alpha} {v}\|^2_{L^2(\O)}\bigg)^{\frac{1}{2}}$ and semi norm  $| {v} |_{H^m(\O)} := \bigg(\underset{|\alpha| = m}{\sum} \|\partial^{\alpha} {v}\|^2_{L^2(\O)}\bigg)^{\frac{1}{2}},$  respectively.  Here,  
 $\alpha= (\alpha_1, \alpha_2)$ represents the multi index in $\mathbb{N}^2$ and the symbol $\partial^{\alpha}$ refers to the partial derivative of $v$ defined as $\partial^{\alpha}v := \frac{\partial^{|\alpha|}v}{\partial x^{\alpha_1} \partial y^{\alpha_2} }$.  Furthermore,  for a non-integer positive number $s$,  we define fractional order Sobolev space $H^{s}(\O) $ as 
\begin{align*}
 H^{s}(\O) :=\bigg \{\b{v}\in H^{\gamma}(\O) : \dfrac{{\lvert{v}(x)−{v}(y)\rvert}}{{\lvert x−y \rvert}^{\theta+1}} \in L^p(\Omega \times \Omega)\bigg \},
\end{align*}
where $s=~\gamma~+~\theta$,   $\gamma$ and $\theta$ denotes the greatest integer and fractional part of $s$,  respectively.   In addition,  for any vector $\b{v}=(v^1,v^2) \in \b{H^{m}}(\O)$~=~$[H^m(\O)]^2$, we define the  product norm as $\|\b{v}\|_{\b{H^m}(\O)}:=\bigg(\sum\limits_{i=1}^2\|{v^i}\|^2_{{H^m(\O)}}\bigg)^{\frac{1}{2}} $ and  semi norm  $|\b{v}|_{\b{H^m}(\O)}:=\bigg(\sum\limits_{i=1}^2|{v^i}|^2_{{H^m(\O)}}\bigg)^{\frac{1}{2}}$.  
\par
\noindent
 In unilateral contact problem,  the contact between two bodies occurs on the part of the boundary  thus the trace operator plays an important role.   We refer to \cite{BScott:2008:FEM,  KO:1988:CPBook} for the detailed understanding of trace operators.  
 Let  the trace of $\b{H^1}(\O)$ functions restricted to boundary $\Gamma_C$ is denoted by $\b{H^{\frac{1}{2}}}(\Gamma_C)$ which is equipped with the norm 
  \begin{equation*}
  \|\b{v}\|_{\b{H^{\frac{1}{2}}}(\Gamma_C) }: =\bigg (\|\b{v}\|^2_{\b{L^2
  }(\Gamma_C)} +  \int_{\Gamma_C}\int_{\Gamma_C} \dfrac{\lvert{\b{v}(x)−\b{v}(y)\rvert^2}}{{\lvert x−y \rvert}^{2}}~dx dy \bigg)^{\frac{1}{2}},~\quad~\b{v}\in {\b{H^{\frac{1}{2}}}(\Gamma_C) }.
 \end{equation*}

 \noindent
Further, let $\b{\gamma_c} : \b{H^1}(\O) \rightarrow \b{H^{\frac{1}{2}}}(\Gamma_C)$ be the trace map.  The continuity and surjectivity of the trace map guarantees existence of continuous right inverse which we denote by $\b{\hat{\gamma}_c}:   \b{H^{\frac{1}{2}}}(\Gamma_C)  \rightarrow  \b{H^1}(\O)$ such that $\b{{\gamma_c}~o~\hat{\gamma}_c}=\text{identity}$.
Denote the space $\b{H^{-\frac{1}{2}}}(\Gamma_C)$ as the dual space of  $\b{H^{\frac{1}{2}}}(\Gamma_C)$ endowed with the norm
\begin{align*}
\|\b{T}\|_{\b{H^{-\frac{1}{2}}}(\Gamma_C)} := \underset{\b{v} ~\in~\b{H^{\frac{1}{2}}}(\Gamma_C),~\b{v}\neq 0}{\text{sup}} \frac{\b{T}\b{(v)}}{~~\|\b{v}\|_{\b{H^{\frac{1}{2}}}(\Gamma_C)}}~\forall~\b{T}\in {\b{H^{-\frac{1}{2}}}(\Gamma_C) },
\end{align*}
and let $\langle \cdot,\cdot \rangle_c$ denotes the duality pairing between the space $\b{H^{-\frac{1}{2}}}(\Gamma_C)$ and $\b{H^{\frac{1}{2}}}(\Gamma_C)$ i.e.  $\langle \b{T},\b{w} \rangle_c := \b{T}(\b{w})~\forall~\b{T} \in \b{H^{-\frac{1}{2}}}(\Gamma_C)$ and $\b{w} \in \b{H^{\frac{1}{2}}}(\Gamma_C)$.  In the further sections,  the conventional symbols are used without any specific explanation.
\subsection{Variational Setting}\label{weak}
In this subsection, we introduce the weak formulation of the Signorini problem which is represented by a variational inequality due to the constraints on the potential contact boundary.  For this purpose,  let us introduce the Hilbert space $\b{{V}}$ for the displacement fields as
\begin{align*}
\b{{V}}:= \{ \b{v} \in \b{H^1}(\O):   \b{v} = \b{0}~ \text{on}~\Gamma_D \},
\end{align*}
 and let $\b{{V}^{'}}$ denotes its topological dual space endowed with the usual norm.  Further,  in order to incorporate the unilateral contact conditions on $\Gamma_C$ we define a non-empty,  closed convex set of admissible displacements as
\begin{align*}
\b{{K}}:= \{ \b{v} \in \b{{V}}:  v_n \leq 0~ \text{on}~\Gamma_C \}.
\end{align*} 
Exploiting the relation $\sigma_n(\b{u})(v_n-u_n)\geq 0~\forall~\b{v} \in \b{{K}} ~ \text{on}~\Gamma_C$ and using integration by parts,  we derive the following variational formulation: find a displacement field $\b{u} \in \b{{K}}$ such that
\begin{align}\label{weak1}
a(\b{u}, ~\b{v}-\b{u}) \geq  L(\b{v}- \b{u})~ \qquad\forall~\b{v} \in \b{{K}},
\end{align} 
where the continuous,  $\b{{V}}$-elliptic bilinear form $a(\cdot,~\cdot)$ and the continuous linear functional $L(\cdot)$ are defined as
\begin{align*}
a(\b{w},\b{v}):= \int_{\O}\b{\sigma}(\b{w}):\b{\varepsilon}(\b{v})~\b{dx} ~~~\quad \forall~\b{w}, \b{v}\in\b{{V}}
\end{align*}
\text{and}
\begin{align*}
~~~~~~L(\b{v}):= \int_{\O}\b{f}\cdot\b{v}~\b{dx} + \int_{\Gamma_N}\b{g} \cdot \b{v}~\b{ds}~~~\quad\forall~\b{v}\in\b{{V}}.
\end{align*}
The existence and uniqueness of the solution of the variational inequality \eqref{weak1} follows from the well known theorem of Lions and Stampacchia \cite {AH:2009:VI, KS:2000:VI}.   Following the approach of the article \cite{TG:2016:VIDG1},  we define Lagrange multiplier $\b{\lambda} \in  \b{H^{-\frac{1}{2}}}(\Gamma_C)$ which reflects the residual of $\b{u}$ with respect to continuous variational formulation  \eqref{weak1}. 
\begin{lemma}
Let $\b{u}$ be the solution of the variational formulation  \eqref{weak1}.  Then,  there exists a continuous linear map 
\begin{align}\label{lagrange}
\langle \b{\lambda},  \b{v} \rangle_c:= L(\b{\hat{\gamma}_c}(\b{v})) -  a(\b{u}, \b{\hat{\gamma}_c}(\b{v}))~ \quad\forall~\b{v}~\in~\b{H^{\frac{1}{2}}}(\Gamma_C). ~~~~~
\end{align}
In addition,  it holds that 
\begin{align}\label{lagrange1}
\langle \b{\lambda}, ~\b{{\gamma}_c}(\b{v}) \rangle_c =    L(\b{v}) -  a(\b{u}, \b{v})~\quad\forall~\b{v}\in \b{{V}},
\end{align}
and
\begin{align}\label{lagrange2}
\langle \b{\lambda}, ~\b{{\gamma}_c}(\b{v}) - \b{{\gamma}_c}(\b{u}) \rangle_c \leq \b{0}~\quad\forall~\b{v}\in \b{{K}}.
\end{align}
\end{lemma}
\noindent
For any $\b{v} = (v^1, v^2)=(v^1,0)+(0,v^2) \in  \b{{V}}$,   we will make use of the following representation of Lagrange multiplier $\b{\lambda}=(\lambda^1, \lambda^2)$ in each of its component $\langle \b{\lambda},  \b{v}\rangle_c = \langle \lambda^1, v^1\rangle + \langle \lambda^2, v^2\rangle $ as follows 
\begin{align}
\langle \lambda^1, v^1\rangle : = L((v^1,0)) - a(\b{u},  (v^1,0)); \label{eqnnn1}\\ 
\langle \lambda^2, v^2\rangle : = L((0,v^2)) - a(\b{u},  (0,v^2));\label{eqnnn2}
\end{align}
In the last equations \eqref{eqnnn1} and \eqref{eqnnn2},   $\langle \cdot,\cdot \rangle$ refers to the duality pairing between $H^{-\frac{1}{2}}(\Gamma_C)$ and $H^{\frac{1}{2}}(\Gamma_C)$.
\section{\textbf{Discrete Problem}}
\subsection{Basic Preliminaries and Definitions} \label{sec:Prelim}
In this section,  we recall some basic notations which will be useful in developing the subsequent analysis.
\begin{itemize}[noitemsep]
    \item  $\mathcal{T}_h$ := regular simplical triangulation of domain $\O$,
    \item  $K$:= an element in $\mathcal{T}_h$,
    \item $\mathcal{V}_K$ := set of all vertices of the triangle $K$,
    \item $\mathcal{M}_K$ := set of midpoints of the triangle $K$,
    \item $\mathcal{V}_h$ := set of all vertices of $\mathcal{T}_h$,
      \item $\mathcal{M}_h$ := set of all midpoint of edges in $\mathcal{T}_h$,
       \item $\mathcal{V}_e$ := set of two vertices of edge $e$.
     \item $\mathcal{M}_e$ := midpoint of edge $e$.
    \item $\mathcal{E}_h$ := set of all edges of  $\mathcal{T}_h$,
    \item $\mathcal{E}_h^i$ := set of all interior edges of $\mathcal{T}_h$,
    \item $\mathcal{E}_h^D$:= set of all edges lying  on $\Gamma_D$,
    \item $\mathcal{E}_h^C$:= set of all edges lying on $\Gamma_C$,
    \item $\mathcal{E}_h^0$ :=  $\mathcal{E}_h^i \cup  \mathcal{E}_h^D$,
    \item $\mathcal{T}_h^C$ := set of all triangles $K$ such that $\partial K \cap \Gamma_C$ is non-empty.
     \item $\mathcal{E}_K^C$:= set of contact edges of the triangle $K\in \mathcal{T}_h^C$. 
    \item $\mathcal{M}_K^C$:= set of midpoint of the edges $e \in \mathcal{E}_K^C$. 
        \item $\mathcal{V}_h^D$:= set of all vertices lying on $\overline{\Gamma}_D$, 
       \item $\mathcal{M}_h^D$:= set of midpoints of edges lying on $\Gamma_D$,
        \item $\mathcal{V}_h^N$:= set of all vertices lying on ${\Gamma}_N$,
       \item $\mathcal{M}_h^N$:= set of midpoints of edges lying on $\Gamma_N$,
      \item $\mathcal{V}_h^C$:= set of all vertices lying on $\overline{\Gamma}_C$,
       \item $\mathcal{M}_h^C$:= set of midpoints of edges lying on $\Gamma_C$,
     \item $\mathcal{E}_p$:= set of all edges sharing the node $p$,
     \item $\mathcal{T}_p$:= set of all triangles sharing the node $p$,
     \item $\mathcal{T}_e$:= set of all triangles sharing the edge $e$,
   \item  $h_K$ := diameter of the triangle $K$,
    \item $h$:=  max $\{h_K: K \in \mathcal{T}_h\}$,
   \item  $h_e$:= length of an edge $e$,
     \item $P_{k}(K)$ denotes the space of polynomials of degree $\leq k$ defined on $K$ where $0 \leq k \in \mathbb{Z}$,
\item $|S|$ denotes the cardinality of the set $S$,
\item $X^{'}$ := dual space of a Banach space $X$.
\end{itemize}
In order to formulate the  discontinuous Galerkin methods for the weak formulation $\eqref{weak1}$ conveniently,   we define the following broken Sobolev space 
\begin{align*}
\b{H^1}(\O, \mathcal{T}_h):= \{ \b{v} \in \b{L^2}(\O)|~ \b{v}_K:=\b{v}|_{K} \in \b{H^1}(K)~\forall~K~\in~\mathcal{T}_h \},
\end{align*}
with the  broken Sobolev norm $\|\b{v}\|^2_{\b{h}}=\underset{K \in \mathcal{T}_h}{\sum} \| \b{v}\|^2_{\b{H^1}(K)}~\forall~\b{v} \in \b{H^1}(\O, \mathcal{T}_h)$. 
\par
\noindent
  For a  scalar valued function $w \in {H^1}(\O, \mathcal{T}_h) $,  vector valued function $\b{v} \in \b{H^1}(\O, \mathcal{T}_h)$ and tensor valued function $\b{\phi} \in [{H^1}(\O, \mathcal{T}_h)]^{2 \times 2} $ which are double valued across the inter element boundary $e\in \mathcal{E}_h^i$,  the jumps and averages across the edge $e$ are defined as
\begin{align*}
\smean{w}&:= \dfrac{w|_{K_1} + w|_{K_2}}{2},~~~\sjump{w}:= w|_{K_1}\b{n_1} + w|_{K_2}\b{n_2} ; \\
\smean{\b{v}}&:= \dfrac{\b{v}|_{K_1} + \b{v}|_{K_2}}{2},~~~~\sjump{\b{v}}:= \b{v}|_{K_1}\otimes\b{n_1} + \b{v}|_{K_2}\otimes\b{n_2}; \\
\smean{\b{\phi}}&:= \dfrac{\b{\phi}|_{K_1} + \b{\phi}|_{K_2}}{2},~~~\sjump{\b{\phi}}:= \b{\phi}|_{K_1}\b{n_1} + \b{\phi}|_{K_2}\b{n_2};
\end{align*}
where the edge $e$ is shared by two contiguous elements $K_1$ and $K_2$,  wherein $\b{n_1}$ is the outward unit normal vector pointing from $K_1$ to $K_2$ and $\b{n_2} = -\b{n_1}$.  Further,  for $\b{x}, \b{y} \in \mathbb{R}^2$,  the dyadic product $\b{x} \otimes \b{y} \in \mathbb{R}^{2 \times 2}$ is defined as $(\b{x} \otimes \b{y})_{ij}= x_iy_j$.  For the convenience,  we also define the jump and averages on the boundary edge $e\in \mathcal{E}_h^b$ as
\begin{align*}
\smean{w}&:=w,~~~~~~~~~\sjump{w}:= w\b{n_e} ; \\
\smean{\b{v}}&:= \b{v},~~~~~~~~~~\sjump{\b{v}}:= \b{v}\otimes\b{n_e} ; \\
\smean{\b{\phi}}&:=  \b{\phi},~~~~~~~~~\sjump{\b{\phi}}:= \b{\phi}\b{n_e} ;
\end{align*}
where $K \in \mathcal{T}_h$ is such that $e \subset \partial K$ and $\b{n_e}$ is the unit normal on the edge $e$ pointing outside $K$.  Hereafter,  the notations $\b{\varepsilon}_h(\b{v})$ and $\b{\sigma}_h(\b{v})$ are defined by the following relation $\b{\varepsilon}_h(\b{v}) := \b{\varepsilon}(\b{v}) $ and $\b{\sigma}_h(\b{v}) :=\b{\sigma}(\b{v})$ on any triangle  $K \in \mathcal{T}_h$.  Finally the notation $X \lesssim Y$  means there exists a positive generic constant $C$ such that $X \leq CY$.
\par
\vspace{0.2 cm}
\noindent
For the discretization,  we use discontinuous quadratic finite element space associated with the simplicial triangulation $\mathcal{T}_h$ which is defined as follows
\begin{align*}
\b{{V}_h}:= \bigg \{ \b{v_h} \in \b{L^2}(\O):~ \b{v_h}|_{K} \in[P_2(K)]^2 ~~\forall ~K~\in~\mathcal{T}_h \bigg \}.
\end{align*}
Further,  incorporating the unilateral condition in the form of integral constraints,  we define the discrete counterpart of the set of admissible displacements $\b{{K}}$ as
\begin{align*}
\b{{K}_h} : = \bigg\{ \b{v_h}=(v_h^1, v_h^2) \in \b{{V}_h}: \int_e v_{hn}~ds \leq 0 ~\forall~e\in\mathcal{E}_h^C \bigg \}.
\end{align*}
In order to incorporate unilateral contact condition in conforming setting,  one needs to enforce non-penetration condition $u_{hn} \leq 0$ everywhere on $\Gamma_C$ but the major drawback of this model arises in numerically implementing it.  This motivates us to enforce the non-penetration condition in the form of integral constraints.  
For the ease of analysis,  the further study is carried presuming the following assumptions \textbf{(A)} and \textbf{(B)}.\\
\textbf{Assumption (A)} The outward unit normal vector to $\Gamma_C$ is constant and for simplicity we set it to be $\b{e_1}$ where $\b{e_1}$ and $\b{e_2}$ denotes the standard ordered basis functions of $\mathbb{R}^2$.  Thus,  the discrete set $\b{{K}_h}$ reduces to 
\begin{align*}
\b{{K}_h}  = \bigg\{ \b{v_h}=(v_h^1, v_h^2) \in \b{{V}_h}: \int_e v^1_{h}~ds \leq 0 ~\forall~e\in\mathcal{E}_h^C \bigg \}.
\end{align*} 
\textbf{Assumption (B)} Each triangle $K \in \mathcal{T}^C_h$ has exactly one potential contact boundary edge.
\begin{remark}\label{rem1}
In the view of assumption \textbf{(A)},  it can be easily verified that $\langle \lambda^2,  v^2 \rangle=0$ for any $\b{v}=(v^1,v^2) \in \b{V}$ using equation \eqref{eqnnn2} and the continuous variational inequality \eqref{weak1}.
\end{remark}
\par
\vspace{0.2 cm}
\noindent
In this work,  we also require a quadratic conforming finite element space $\b{{V}_c}= \b{{V}_h} \cap \b{{V}}$ associated with the triangulation $\mathcal{T}_h$.   Next,  we revisit the famous Cl$\acute{e}$ment approximation result \cite[Section 4.8,  Pg. 122]{BScott:2008:FEM} which will be helpful in further analysis.
\begin{lemma}\label{clement}
Let $\b{v} \in \b{{V}}$.  Then,  there exists $\b{v_h} \in \b{{V}_c}$ such that the following estimate holds on any triangle $K\in \mathcal{T}_h$
\begin{align*}
\|\b{v}-\b{v_h}\|_{\b{H^l}(K)} \lesssim h^{1-l}_K\|\b{v_h}\|_{\b{H^1}(\mathcal{T}_K)},~~l=0,1,
\end{align*}
where $\mathcal{T}_K$ refers to the set of triangles $\tilde{K} \in \mathcal{T}_h$ such that $\overline{\tilde{K}} \cap \overline{K} \neq \emptyset$.
\end{lemma}
\par 
\noindent
Next,  we state the discrete trace and inverse inequalities \cite{BScott:2008:FEM,Ciarlet:1978:FEM} which will be frequently used in the convergence analysis ahead.
\begin{itemize}
\item \textbf{Discrete trace inequality}:  For all $K \in \mathcal{T}_h$ and $e \subset \partial K$,  the following inequality holds for all $\b{v} \in \b{H^1}(K)$.
\end{itemize}
\begin{align*}
\|\b{v}\|^2_{\b{L^2}(e)} \lesssim h^{-1}_e \|\b{v}\|^2_{\b{L^2}(K)} +h_e |\b{v}|^2_{\b{H^1}(K)}.
\end{align*}
\begin{itemize}
\item \textbf{Inverse inequalities}:  For all $K \in \mathcal{T}_h$ and $e \subset \partial K$,  the inequalities
\begin{align*}
\lVert \b{v_h} \rVert_{\b{L^2}{(e)}} &\lesssim h_e^{-\frac{1}{2}} \lVert \b{v_h} \rVert_{\b{L^2}(K)}, \\
\lVert \nabla \b{v_h} \rVert_{\b{L^2}{(K)}}&\lesssim h_K^{-1} \lVert \b{v_h} \rVert_{\b{L^2}(K)},
\end{align*}
\end{itemize}
~~hold for any $\b{v_h}$ in the discrete space $\b{{V}_h}$.
\subsection{DG Formulation}
This subsection is devoted to state various DG formulations for the continuous variational inequality \eqref{weak1}.  For the sake of brievity,  we first list the bilinear forms $\mathcal{A}_h(\cdot, \cdot)$  corresponding to well known DG methods for the contact problem.  We refer to article \cite{WHC:2011:DGSP} for the complete derivations of these DG formulations.  To this end,  let $\b{r_o}$ and $\b{r_e}$ denotes the global and local lifting operators \cite{ABCM:2002:UnifiedDG},  respectively.  The various DG methods are listed as follows:
\begin{itemize} 
\item \textbf{SIPG method} \cite{Arnold:1982:IPD,  WHC:2011:DGSP}:
\begin{align*}
\hspace{3mm}
\mathcal{A}_{h}^{(1)}(\b{u_h},\b{v_h}) &= \sum\limits_{K \in\mathcal{T}_h}\int\limits_K \b{\sigma}(\b{u_h}):\b{\varepsilon}(\b{v_h})~\textbf{dx} - \int \limits_{e \in \mathcal{E}^0_h } \sjump{\b{u_h}}:\smean{\b{\sigma}_h(\b{v_h})} ~\textbf{ds} - \int \limits_{e \in \mathcal{E}^0_h } \sjump{\b{v_h}}:\smean{\b{\sigma}_h(\b{u_h})}~\textbf {ds} \\
& + \int \limits_{e \in \cE_h^0}\eta h_e^{-1}\sjump{\b{u_h}}:\sjump{\b{v_h}}~\textbf{ds},
\end{align*}
\noindent
for $\b{u_h}, \b{v_h} \in \b{{V}_h}$ and $\eta > \eta_0>0$.
\bigbreak
\item \textbf{NIPG method} \cite{Apriori:2001:Wheeler,WHC:2011:DGSP}
\begin{align*}
\hspace{3mm}
\mathcal{A}_{h}^{(2)}(\b{u_h},\b{v_h}) &= \sum\limits_{K\in\mathcal{T}_h}\int\limits_K \b{\sigma}(\b{u_h}):\b{\varepsilon}(\b{v_h})~ \textbf{dx} + \int \limits_{e \in \mathcal{E}^0_h } \sjump{\b{u_h}}:\smean{\b{\sigma}_h(\b{v_h})} ~\textbf{ds} - \int \limits_{e \in \mathcal{E}^0_h } \sjump{\b{v_h}}:\smean{\b{\sigma}_h(\b{u_h})}~\textbf {ds} \\
& + \int \limits_{e \in \cE_h^0}\eta h_e^{-1}\sjump{\b{u_h}}:\sjump{\b{v_h}}~\textbf{ds},
\end{align*}
for $\b{u_h}, \b{v_h} \in \b{{V}_h}$ and $\eta > 0.$
\bigbreak
\item  \textbf{Bassi et al.} \cite{ Bassi,WHC:2011:DGSP}:
\begin{align*}
\mathcal{A}_{h}^{(3)}(\b{u_h},\b{v_h}) &= \sum\limits_{K\in\mathcal{T}_h}\int\limits_K \b{\sigma}(\b{u_h}):\b{\varepsilon}(\b{v_h})~\textbf{dx}  - \int \limits_{e \in \mathcal{E}^0_h } \sjump{\b{u_h}}:\smean{\b{\sigma}_h(\b{v_h})}~\textbf{ds} - \int \limits_{e \in \mathcal{E}^0_h } \sjump{\b{v_h}}:\smean{\b{\sigma}_h(\b{u_h})}~\textbf {ds}\\&+ \sum\limits_{e \in \cE^0_h}\int_\O \eta \mathcal{C}\b{r}_e(\sjump{\b{u_h}}):\b{r}_e(\sjump{\b{v_h}})~\textbf{dx},
\end{align*} 
for $\b{u_h}, \b{v_h} \in \b{{V}_h}$ and $\eta > 3.$
\bigbreak
\item  \textbf{Brezzi et al.} \cite{BMMPR:2000:BrezziDG,WHC:2011:DGSP}:
\begin{align*}
\mathcal{A}_{h}^{(4)}(\b{u_h},\b{v_h}) &= \sum\limits_{K\in\mathcal{T}_h}\int\limits_K \b{\sigma}(\b{u_h}):\b{\varepsilon}(\b{v_h})~\textbf{dx}  - \int \limits_{e \in \mathcal{E}^0_h } \sjump{\b{u_h}}:\smean{\b{\sigma}_h(\b{v_h})}~\textbf{ds} - \int \limits_{e \in \mathcal{E}^0_h } \sjump{\b{v_h}}:\smean{\b{\sigma}_h(\b{u_h})}~\textbf {ds}\\&+  \int \limits_\O  \b{r}_o(\sjump{\b{v_h}}):\mathcal{C}\b{r}_o(\sjump{\b{u_h}})~\textbf{dx} + \sum\limits_{e \in \cE^0_h}\int_\O \eta \mathcal{C}\b{r}_e(\sjump{\b{u_h}}):\b{r}_e(\sjump{\b{v_h}})~\textbf{dx},
\end{align*} 
for $\b{u_h}, \b{v_h} \in \b{{V}_h}$ and $\eta > 0.$
\bigbreak
\item  \textbf{LDG Method} \cite{CCPS:2000:LDG,  Chen::2010:ldg, WHC:2011:DGSP}:
\begin{align*}
\mathcal{A}_{h}^{(5)}(\b{u_h},\b{v_h}) &= \sum\limits_{K\in\mathcal{T}_h}\int\limits_K \b{\sigma}(\b{u_h}):\b{\varepsilon}(\b{v_h})~\textbf{dx}  - \int \limits_{e \in \mathcal{E}^0_h } \sjump{\b{u_h}}:\smean{\b{\sigma}_h(\b{v_h})}~\textbf{ds} - \int \limits_{e \in \mathcal{E}^0_h } \sjump{\b{v_h}}:\smean{\b{\sigma}_h(\b{u_h})}~\textbf {ds}\\&+ \int \limits_\O \b{r}_o(\sjump{\b{v_h}}):\mathcal{C}\b{r}_o(\sjump{\b{u_h}})~\textbf{dx}+ \int \limits_{e \in \cE_h^0}\eta h_e^{-1}\sjump{\b{u_h}}:\sjump{\b{v_h}}~\textbf{ds} ,
\end{align*} 
for $\b{u_h}, \b{v_h} \in \b{{V}_h}$ and $\eta > 0.$
\end{itemize}
Finally,  we are in the position to introduce the discrete problem which reads as follows: find $\b{u_h} \in \b{{K}_h}$ such that
\begin{align}\label{discrete}
\mathcal{A}_h(\b{u_h}, \b{v_h}-\b{u_h}) \geq L(\b{v_h}-\b{u_h}) ~\quad\forall~\b{v_h} \in \b{{K}_h},
\end{align}
where $\mathcal{A}_h(\b{u_h}, \b{v_h}) $ represents one of the bilinear form $\mathcal{A}_h^{(i)}(\b{u_h}, \b{v_h}),~1\leq i \leq 5$.  Note that, the discrete bilinear form can be reformulated as 
\begin{align}\label{reformulate}
\mathcal{A}_h(\b{u_h}, \b{v_h}) = a_h(\b{u_h}, \b{v_h}) + b_h(\b{u_h}, \b{v_h}),
\end{align}
where 
\begin{align*}
a_h(\b{u_h}, \b{v_h}) = \sum\limits_{K\in\mathcal{T}_h}\int\limits_K \b{\sigma}(\b{u_h}):\b{\varepsilon}(\b{v_h})~\textbf{dx},
\end{align*}
and $b_h(\b{u_h}, \b{v_h})$ consists of all the remaining (consistency and stability) terms.  We note  the following  bound on the bilinear form $b_h(\cdot, \cdot)$ for the various choices of DG methods listed above.
\begin{align}\label{bound3}
|b_h(\b{w_h},  \b{v_h})| \lesssim \bigg( \sum\limits_{e \in \mathcal{E}_h^0} \frac{1}{h_e}\| \sjump {\b{w_h}} \|^2_{\b{L^2}(e)}\bigg)^{\frac{1}{2}}|\b{v_h}|_{\b{H^1}(\Omega)}~\quad \forall~\b{w_h}\in \b{{V}_h}, ~\b{v_h}\in \b{{V}_c}.
\end{align}
\par 
\noindent
The next task is to define DG norm on the discrete space $\b{{V}_h}$,  for which we define the following relations: for any $K \in \mathcal{T}_h$
\begin{align}\label{relation}
|\b{v_h}|^2_K= \int \limits_K \b{\sigma}(\b{u_h}) : \b{\varepsilon}(\b{v_h})~\b{dx}, ~~~~~~|\b{v_h}|^2_{\b{h}}= \sum\limits_{K \in \mathcal{T}_h}|\b{v_h}|^2_K, ~~~~~~
|\b{v_h}|^2_{\b{*}}= \sum\limits_{e \in \mathcal{E}_h^0}\frac{1}{h_e} \| \sjump {\b{v_h}} \|^2_{\b{L^2}(e)}.
\end{align}
We now define DG norm on space $\b{{V}_h} + \b{{V}}$ as 
\begin{align*}
\b{\norm{\b{v}}}^2_{\b{h}}:= |\b{v}|^2_{\b{h}} + |\b{v}|^2_{\b{*}} + \sum\limits_{e \in \mathcal{E}_h^0}h_e\|\smean{\b{\varepsilon}_h(\b{v})}\|^2_{\b{L^2}(e)} ~\forall~\b{v}\in \b{{V}_h} + \b{{V}}.
\end{align*} 

\noindent
Note that the norm $\norm{\cdot}_{\b{h}}$ is equivalent to the usual DG norm$\bigg(\b{\|\b{v}\|}^2_{\b{h}}+\lvert \b{v}\rvert^2_* + \sum\limits \limits_{e \in \mathcal{E}_h^0}h_e\|\smean{\b{\varepsilon}_h(\b{v}}\|^2_{\b{L^2}(e)}\bigg)^{\frac{1}{2}}$ by Korn's-inequality and Poincar$\acute{e}$-Fredrichs inequality for piecewise $H^1$ spaces \cite{ Brenner::2004:Korn, Brenner:2001:Poincare}.
Observe that with respect to $\norm{\cdot}_{\b{h}}$,  the bilinear form $\mathcal{A}_h(\cdot,\cdot)$ is bounded over the space $\b{{V}_h}+\b{{V}},$ i.e.
\begin{align*}
\mathcal{A}_h(\b{v},\b{w}) \leq \norm{\b{v}}_{\b{h}}\norm{\b{w}}_{\b{h}}~~\quad\forall ~\b{v},\b{w}~\in~ \b{{V}_h}+ \b{{V}},
\end{align*}
and is coercive over the space $\b{{V}_h},$ i.e.
\begin{align*}
\mathcal{A}_h(\b{v_h},\b{v_h}) \geq \alpha \norm{\b{v_h}}^2_{\b{h}}~\quad \forall ~\b{v_h}~\in~ \b{{V}_h}.
\end{align*}
{
where $\alpha>0$ is the $\b{V}$- ellipticity constant. \\}

\noindent
Analogous to the continuous problem \eqref{weak1},  the discrete problem \eqref{discrete} also possess a unique solution owing to the result of Stampacchia.  In the next subsection,  we  introduce some additional functional tools which will be helpful in  developing further analysis. 
\subsection{Interpolation Operators}
 We will define couple of interpolation operators which will be crucially used in establishing further results. 
 \begin{itemize}
 \item Define the interpolation operator  $\b{\mathcal{I}_h}: \b{{V}} \cap [C(\overline{\O})]^2 \longrightarrow \b{{V}_h}$ as
\begin{align*}
\b{\mathcal{I}_h} \b{v} := \b{\mathcal{I}}_K\b{v}~~\quad\forall~K~\in~\mathcal{T}_h,
\end{align*}
where,  the operator $\b{\mathcal{I}}_K\b{v}$ is given by 
\begin{itemize}
\item If $K \in \mathcal{T}_h^C$,  define
\begin{align*}
\b{\mathcal{I}}_K\b{v}(p) &= \b{v}(p)~ \quad\forall ~p \in (\mathcal{V}_K \cup \mathcal{M}_K)\backslash \mathcal{M}_K^C,  \\
\int_e \b{\mathcal{I}}_K\b{v} &= \int_e \b{v},\quad ~e\in\mathcal{E}^C_K.  
\end{align*}
otherwise,
\begin{align*}
\b{\mathcal{I}}_K\b{v}(p) &= \b{v}(p)~\quad \forall ~p \in (\mathcal{V}_K \cup \mathcal{M}_K).  
\end{align*}
\end{itemize} 
 \end{itemize}
It can be observed that the interpolation operator $\b{\mathcal{I}_h}$ is invariant for any $\b{v}\in [P_2(K)]^2$.  Thus,  in view of  Bramble Hilbert Lemma \cite{Ciarlet:1978:FEM},  we determine the following approximation property of the  map $\b{\mathcal{I}_h}$.
\begin{lemma}\label{Interpolation1}
For any $\b{v} \in \b{H^\nu}(K),  K \in \mathcal{T}_h$,  the following holds
\begin{align*}
\|\b{v}-\b{ \mathcal{I}}_K\b{v}\|_{\b{H^l}(K)} \lesssim  h^{\nu-l}|\b{v}|_{\b{H^{\nu}}(K)}~~~~\text{for}~~0\leq l\leq \nu,
\end{align*}
where $~2\leq\nu\leq3$.
\end{lemma}
\noindent
In order to define next interpolation operator $\b{\pi_h}$,  we need to introduce some more notations related to $\Gamma_C$.  We define the space $\b{M_h}(\Gamma_C)$ as 
\begin{align*}
\b{M_h}(\Gamma_C) := \{ \b{v_h} \in \b{L^2}(\Gamma_C):~\b{v_h}|_e \in [P_0(e)]^2~\forall~e\in \mathcal{E}^C_h \}.
\end{align*}
Next,  with the help of the space $\b{M_h}(\Gamma_C)$, we define a projection operator 
$\b{\pi_h} : ~\b{L^2}(\Gamma_C) \longrightarrow  \b{M_h}(\Gamma_C)$ as $\b{\pi_h}(\b{v}) = \b{\pi_h}(\b{v})|_e~\forall~\b{v}\in\b{L^2}(\Gamma_C),  e\in \mathcal{E}^C_h $ where 
\begin{align}\label{pi_h}
\b{\pi_h}(\b{v})|_e:= \dfrac{1}{h_e} \int \limits_e \b{v}~\textbf{ds}.
\end{align}
We make use of the representation $\b{\pi_h}(\b{v}) = ( \pi^1_h(\b{v}),  \pi^2_h(\b{v}))$ where the components are given by
\begin{align*}
 \pi^1_h(\b{v}) =  \dfrac{1}{h_e} \int \limits_e {v^1}~{ds}, ~~~~~
\pi^2_h(\b{v}) =  \dfrac{1}{h_e} \int \limits_e {v^2}~{ds},
\end{align*}
for all $\b{v}=(v^1,v^2)\in \b{L^2}(\Gamma_C)$.
Further,  it is well known that the interpolation operator $\b{\pi_h}$ satisfies the following approximation properties \cite{BScott:2008:FEM,  Ciarlet:1978:FEM}
\begin{align}\label{prop}
\|\b{v}-\b{\pi_h}(\b{v})\|_{\b{L^2}(\Gamma_C)} \lesssim h^{\nu}\|\b{v}\|_{\b{H^{\nu}}(\Gamma_C)}, ~\quad 0\leq \nu \leq 1, ~\b{v}\in \b{H^{\nu}}(\Gamma_C).
\end{align}
On the similar lines,  for any scalar valued function $w \in L^2(\Gamma_C)$,  we denote $\tilde{\pi}_h(w)$ as the  $L^2$ projection of ${w}$ onto the space of constant functions which in turn fulfils the approximation properties  listed in Lemma \ref{approx_result} (see \cite{apriori:2001:quad}).
\begin{lemma}\label{approx_result}
Let $ 0 \leq \nu_1,  \nu_2\leq 1$,  then for $ w \in H^{\nu_2}(\Gamma_C)$, it holds that
\begin{align*}
\|w-\tilde{\pi}_h({w})\|_{({H^{\nu_1}}(\Gamma_C))^{'}} \lesssim h^{\nu_1 +\nu_2}|{w}|_{{H^{\nu_2}}(\Gamma_C)}.
\end{align*}
\end{lemma}

\section{\textbf{A priori Error Estimates}}
In this section we derive \textit{a priori} error estimates for the error in DG norm under some appropriate regularity assumption on the exact solution $\b{u}$.  The unilateral contact condition give rise to the singular behaviour of $\b{u}$ in the vicinity of free boundary around $\Gamma_C$ even when the forces $\b{f}$ and $\b{g}$ are sufficiently regular.  Keeping this in view,   it shall be a realistic to assume $\b{u} \in \b{H^{2+\epsilon}}(\Omega)$ where $0 <  \epsilon \leq \frac{1}{2}$.  On the contrary,  if the free boundary vanishes or become sufficiently  smooth the contact problem boils down to linear elasticity problem.  The following theorem ensures the optimal order \textit{a priori} error estimates for the Signorini contact problem.
\begin{theorem}\label{thm1}
Let $\b{u} \in \b{H^{2+\epsilon}}(\Omega) \cap \b{K}$ where $0 < \epsilon \leq \frac{1}{2}$ be the solution of continuous problem \eqref{weak1}  and $\b{u_h}$ be the solution of discrete problem \eqref{discrete}.  Then,  the following holds
\begin{align*}
\b{\norm{\b{u-u_h}}_h} \lesssim  h^{1+\epsilon}.
\end{align*}
\end{theorem} 
\noindent
Before proceeding further to prove Theorem \ref{thm1},  we will recall the following lemma which is the key ingredient in deriving this error estimates (see  \cite[Lemma 8.1 ]{apriori:2001:quad}).
\begin{lemma}\label{discrete1}
Let $\beta \in (\frac{1}{2}, 1]$ and $e \in \mathcal{E}_h^C$,  then for  all $~x_o \in e~$,  the following estimate holds
\begin{align*}
\| \Psi - \Psi(x_o)\|_{L^2(e)} \lesssim h^{\beta} \|\Psi\|_{H^{\beta}(e)} ~\forall~\Psi\in H^{\beta}(e).
\end{align*}
\end{lemma}
\begin{proof}[\textbf{Proof of Theorem \ref{thm1}}]
We start with splitting the error into two parts
\begin{align}\label{1}
\norm{\b{u-u_h}}^2_{\b{h}} \lesssim \norm{\b{u-\b{\mathcal{I}_h}\b{u}}}^2_{\b{h}} + \norm{\b{\mathcal{I}_h}\b{u}-\b{u_h}}^2_{\b{h}}.
\end{align}
Using the definition of norm  $\norm{\cdot}_{\b{h}}$,  we have
\begin{align}
\norm{\b{u-\b{\mathcal{I}_h}\b{u}}}^2_{\b{h}} \lesssim \|\b{u-\b{\mathcal{I}_h}\b{u}}\|^2_{\b{h}} +\mid \b{u-\b{\mathcal{I}_h}\b{u}}\mid^2_* + \sum\limits \limits_{e \in \mathcal{E}_h^0}h_e\|\smean{\b{\varepsilon}_h(\b{u-\b{\mathcal{I}_h}\b{u}})}\|^2_{\b{L^2}(e)}
\end{align}

In the view of Lemma \ref{Interpolation1}, discrete trace inequality and the fact that $\sjump{\b{u}-\b{\mathcal{I}_h}\b{u}}=\b{0}$ on interelement boundaries,  we have the following approximation property
 \begin{align}\label{2}
 \norm{\b{u-\b{\mathcal{I}_h}\b{u}}}^2_{\b{h}} \lesssim  h^{2(1+\epsilon)}\|\b{u}\|^2_{\b{H^{2+\epsilon}}(\Omega)}.
\end{align}
In order to bound $\norm{\b{\mathcal{I}_h\b{u}-\b{u_h}}}^2_{\b{h}}$,  we use the stability of the bilinear form $\mathcal{A}_h(\cdot,\cdot) $ with respect to DG norm $\norm{\cdot}_{\b{h}}$ for $\b{v_h} \in \b{{V}_h}$ as follows
\begin{align}\label{8}
\alpha \norm{\b{\mathcal{I}_h}\b{u}-\b{u_h}}^2_{\b{h}}  &\leq \mathcal{A}_h(\b{\mathcal{I}_h}\b{u}-\b{u_h},\b{\mathcal{I}_h}\b{u}-\b{u_h}) \nonumber\\
&= \mathcal{A}_h(\b{\mathcal{I}_h}\b{u}-\b{u},\b{\mathcal{I}_h}\b{u}-\b{u_h})+\mathcal{A}_h(\b{u},\b{\mathcal{I}_h}\b{u}-\b{u_h})-\mathcal{A}_h(\b{u_h},\b{\mathcal{I}_h}\b{u}-\b{u_h}) \nonumber \\
&= T_1 +T_2 +T_3,
\end{align}
where,
\begin{align*}
T_1 &:= \mathcal{A}_h(\b{\mathcal{I}_h}\b{u}-\b{u},\b{\mathcal{I}_h}\b{u}-\b{u_h}), \\
T_2&:= \mathcal{A}_h(\b{u},\b{\mathcal{I}_h}\b{u}-\b{u_h}), \\
T_3&:=-\mathcal{A}_h(\b{u_h},\b{\mathcal{I}_h}\b{u}-\b{u_h}).
\end{align*}
Next,  we  bound $T_1, ~T_2 ~ \text{and} ~T_3$ individually.  Using the continuity of bilinear form $\mathcal{A}_h(\cdot, \cdot)$ and Young's inequality,  we obtain
\begin{align*}
T_1= \mathcal{A}_h(\b{\mathcal{I}_h}\b{u}-\b{u},\b{\mathcal{I}_h}\b{u}-\b{u_h}) &\leq \norm{\b{\mathcal{I}_h}\b{u}-\b{u}}_{\b{h}}\norm{\b{\mathcal{I}_h{u}}-\b{u_h}}_{\b{h}} \\
& \leq \frac{1}{2\alpha} \norm{\b{\mathcal{I}_h}\b{u}-\b{u}}^2_{\b{h}} + \frac{\alpha}{2} \norm{\b{\mathcal{I}_h{u}}-\b{u_h}}^2_{\b{h}},
\end{align*}
where $\alpha > 0$ denotes the $\b{V}$-ellipticity constant.
Finally,  using \eqref{2}, we find 
\begin{align}\label{3}
T_1 \leq  \frac{1}{2\alpha}h^{2(1+\epsilon)}\|\b{u}\|^2_{\b{H^{2+\epsilon}}(\O)} +\frac{\alpha}{2}\norm{\b{\mathcal{I}_h}{\b{u}}-\b{u_h}}^2_{\b{h}}.
\end{align}
Since $\b{u} \in \b{H}^{\b{2+\epsilon}}(\Omega)$,  it results that $\sjump{\b{u}}$ and $\sjump{\b\sigma(\b{u})}$ vanishes on interelement boundaries together with $\smean{\b{\sigma}(\b{u})} = \b{\sigma}(\b{u})$.  Consequently
\begin{align}\label{4}
T_2 = \sum\limits_{K\in \mathcal{T}_h} \int \limits_K \b{\sigma}(\b{u}):\b{\varepsilon}(\b{\b{\mathcal{I}_h}u}-\b{u_h})~\textbf{dx}- \sum\limits_{e \in \mathcal{E}_h^0} \int \limits_e \b{\sigma}_h(\b{u}) : \sjump{\b{\mathcal{I}_h}\b{u}-\b{u_h}}~\textbf{ds}.
\end{align}
Further we use integration by parts to handle the first term of $\eqref{4}$ as follows
\begin{align}\label{5}
\sum\limits_{K\in \mathcal{T}_h} \int_K \b{\sigma}(\b{u}): \b{\varepsilon}(\b{\mathcal{I}_h}\b{u}-\b{u_h})~\textbf{dx} &=  L(\b{\mathcal{I}_h}\b{u}-\b{u_h}) + \sum\limits_{e \in \mathcal{E}_h^0} \int \limits_e \b{\sigma}_h(\b{u}):\sjump{ \b{\mathcal{I}_h}\b{u}-\b{u_h}}~\textbf{ds} \nonumber \\&+\sum\limits_{e \in \mathcal{E}_h^C} \int \limits_e (\b{\sigma}_h(\b{u})\b{n}) \cdot (\b{\mathcal{I}_h}\b{u}-\b{u_h})~\textbf{ds}.
\end{align}
Inserting  \eqref{5} in \eqref{4},  we obtain
\begin{align}\label{6}
T_2 = L(\b{\mathcal{I}_h}\b{u}-\b{u_h})+\sum\limits_{e \in \mathcal{E}_h^C} \int \limits_e (\b{\sigma}_h(\b{u})\b{n}) \cdot (\b{\mathcal{I}_h}\b{u}-\b{u_h})~\textbf{ds}.
\end{align}
A use of discrete variational inequality \eqref{discrete} and the fact that $\b{\mathcal{I}_h}\b{u} \in \b{{K}_h}$, we find
\begin{align}\label{7}
T_3 \leq -L(\b{\mathcal{I}_h}\b{u}-\b{u_h}).
\end{align}
Thus,  using \eqref{8},  \eqref{3},  \eqref{6} and \eqref{7}     together with the decomposition formula \eqref{decomp} and assumption \textbf{(A)},  we find
\begin{align} \label{10}
\frac{\alpha}{2} \norm{\b{\mathcal{I}_h}\b{u}-\b{u_h}}^2_{\b{h}}\leq \frac{1}{2\alpha}h^{2(1+\epsilon)}\|\b{u}\|^2_{\b{H^{2+\epsilon}}(\Omega)} + \sum\limits_{e \in \mathcal{E}_h^C} \int\limits_e \sigma_n(\b{u})({\mathcal{I}_h^1{\b{u}}}-{u^1_h})~ds,
\end{align}
where $\b{\mathcal{I}_h\b{u}} = (\mathcal{I}^1_h{\b{u}}, \mathcal{I}^2_h{\b{u}})$,  each component is defined in a usual way.
Further,  we estimate the second  term of \eqref{10} as follows
\begin{align}\label{11}
\sum\limits_{e \in \mathcal{E}_h^C} \int \limits_e \sigma_n(\b{u})({\mathcal{I}_h^1\b{u}}-{u^1_h})~ds &= \sum\limits_{e \in \mathcal{E}_h^C} \int \limits_e \sigma_n(\b{u})({\mathcal{I}_h^1\b{u}}-{u^1})~ds + \sum\limits_{e \in \mathcal{E}_h^C} \int 
\limits_e \sigma_n(\b{u})({{u^1}}-{u^1_h})~ds \nonumber \\
&= Q_1+Q_2,
\end{align}
where 
\begin{align*}
Q_1 &:=  \sum\limits_{e \in \mathcal{E}_h^C} \int \limits_e \sigma_n(\b{u})({\mathcal{I}_h^1\b{u}}-{u^1})~ds, \\
Q_2&:=  \sum\limits_{e \in \mathcal{E}_h^C} \int \limits_e \sigma_n(\b{u})({{u^1}}-{u^1_h})~ds. 
\end{align*}
Now,  we  estimate $Q_1$ and $Q_2$ one by one.  

As $\b{u}\in\b{H^{2+\epsilon}}(\Omega)$,  the trace of $\b{\sigma}(\b{u})$ belongs to $\b{H^{\frac{1}{2}+\epsilon}}(\Gamma_C)$.  In view of the embedding result, $\b{H^{\frac{1}{2}+\epsilon}}(\Gamma_C)\hookrightarrow [{\mathcal{C}}(\Gamma_C)]^2$, the pointwise values of $\b{\sigma}(\b{u})$ are well defined.  To this end,  consider the set $$\mathcal{J}_1 = \{ e \in \mathcal{E}^C_h : \sigma_n(\b{u})~ \text{vanishes atleast at one point on}~ e~\in \mathcal{E}^C_h \}.$$The summation in $Q_1$ reduces to sum all edges in $\mathcal{J}_1$ because if $\sigma_n(\b{u})>0$ on some edge $\tilde{e} \in \mathcal{E}^C_h$.  Then,  using the fact that  $u_1\sigma_n(\b{u})=0$ on $\Gamma_C$, we obtain $u_1=0$ on that edge $\tilde{e}$,  consequently  $\mathcal{I}_h^1\b{u}=0$ on $\tilde{e}$.  Thus,  we have
\begin{align}\label{12}
Q_1 &=\sum\limits_{e \in \mathcal{J}_1} \int \limits_e \sigma_n(\b{u})({\mathcal{I}_h^1\b{u}}-{u^1})~ds.
\end{align} 
A use of Cauchy-Schwarz inequality yields
\begin{align}\label{13}
Q_1 &\leq \sum\limits_{e\in \mathcal{J}_1} \|\sigma_n(\b{u})\|_{L^2(e)}\|{\mathcal{I}_h^1\b{u}}-{u^1}\|_{L^2(e)} \nonumber \\
&\leq \bigg(\sum\limits_{e\in \mathcal{J}_1} \|\sigma_n(\b{u})\|^2_{L^2(e)}\bigg)^{\frac{1}{2}}\bigg(\sum\limits_{e\in \mathcal{J}_1} \|{\mathcal{I}_h^1\b{u}}-{u^1}\|^2_{L^2(e)}\bigg)^{\frac{1}{2}}.
\end{align}
Further,  a use of discrete trace inequality together with Lemma \ref{Interpolation1} in equation \eqref{13} yields
\begin{align}\label{14}
Q_1 &\leq \bigg(\sum\limits_{e\in \mathcal{J}_1} \|\sigma_n(\b{u})\|^2_{L^2(e)}\bigg)^{\frac{1}{2}}\bigg(\sum\limits_{e\in \mathcal{J}_1} \sum\limits_{K\in\mathcal{T}_e}(h^{-1}_K \|{\mathcal{I}_h^1\b{u}}-{u^1}\|^2_{L^2(K)}+h_K |{\mathcal{I}_h^1\b{u}}-{u^1}|^2_{H^1(K)}\bigg)^{\frac{1}{2}} \nonumber\\
&\lesssim \bigg(\sum\limits_{e\in \mathcal{J}_1} \|\sigma_n(\b{u})\|^2_{L^2(e)}\bigg)^{\frac{1}{2}}h^{\frac{3}{2}+\epsilon}\|\b{u}\|_{\b{H}^{\b{2+\epsilon}}(\Omega)}.
\end{align}
Since for any arbitrary edge $e\in \mathcal{J}_1$,  $\sigma_n(\b{u})$ vanishes at atleast one point on that edge,  a use of Lemma $\ref{discrete1}$ gives
\begin{align}\label{15}
\|\sigma_n(\b{u})\|_{L^2(e)} \lesssim {h_e}^{\frac{1}{2}+\epsilon}\|\sigma_n(\b{u})\|_{H^{\frac{1}{2}+\epsilon}(e)}.
\end{align} 
Inserting \eqref{15} into \eqref{14} and using trace theorem \cite{BScott:2008:FEM} gives
\begin{align}\label{30}
Q_1&\leq \bigg(\sum\limits_{e\in \mathcal{J}_1} h_e^{{1+2\epsilon}}\|\sigma_n(\b{u})\|^2_{H^{\frac{1}{2}+\epsilon}(e)}\bigg)^{\frac{1}{2}}h^{\frac{3}{2}+\epsilon}\|\b{u}\|_{\b{H}^{\b{2+\epsilon}}(\Omega)} \nonumber \\
&\leq h^{\frac{1}{2}+\epsilon}\|\sigma_n(\b{u})\|_{H^{\frac{1}{2}+\epsilon}(\Gamma_C)}h^{\frac{3}{2}+\epsilon}\|\b{u}\|_{\b{H}^{\b{2+\epsilon}}(\Omega)} \nonumber \\
&\lesssim h^{\frac{1}{2}+\epsilon}\|\b{u}\|_{\b{H^{2+\epsilon}}(\Omega)}h^{\frac{3}{2}+\epsilon}\|\b{u}\|_{\b{H^{2+\epsilon}}(\Omega)} \nonumber \\
&\lesssim  h^{2(1+\epsilon)}\|\b{u}\|^2_{\b{H^{2+\epsilon}}(\Omega)}.
\end{align}
\par
\noindent 
Consider 
\begin{align}\label{17}
Q_2 &= \sum\limits_{e\in \mathcal{E}^C_h} \int\limits_e \sigma_n(\b{u})(u^1-u^1_h)~ds \nonumber \\
&=  \sum\limits_{e\in \mathcal{E}^C_h} \int\limits_e \big( \sigma_n(\b{u})-\tilde{\pi}_h(\sigma_n(\b{u})\big)(u^1-u^1_h)~ds + \sum\limits_{e\in \mathcal{E}^C_h} \int\limits_e {\tilde{\pi}_h(\sigma_n(\b{u}))}(u^1-u^1_h)~ds.
\end{align}
Using H\"{o}lder's inequality,  Lemma {\ref{approx_result}}, trace theorem and Young's inequality in the first integral of \eqref{17},   we find
\begin{align}\label{new}
 \sum\limits_{e\in \mathcal{E}^C_h} \int_e \big( \sigma_n(\b{u})&-{\tilde{\pi}_h(\sigma_n(\b{u})}\big)(u^1-u^1_h)~ds \nonumber \\ &\leq\bigg( \sum\limits_{e\in \mathcal{E}^C_h}\|\sigma_n(\b{u})-{\tilde{\pi}_h(\sigma_n(\b{u}))}\|^2_{[H^{\frac{1}{2}}(e)]^{'}}\bigg)^{\frac{1}{2}}\bigg( \sum\limits_{e\in \mathcal{E}^C_h}\|u^1-u^1_h\|^2_{H^{\frac{1}{2}}(e)}\bigg)^{\frac{1}{2}} \nonumber \\
&\leq\bigg( h^{2(1+\epsilon)}\sum\limits_{e\in \mathcal{E}^C_h}\|\sigma_n(\b{u})\|^2_{H^{\frac{1}{2}+\epsilon}(e)}\bigg)^{\frac{1}{2}}\bigg( \sum\limits_{e\in \mathcal{E}^C_h}\sum\limits_{K\in\mathcal{T}_e}\|u^1-u^1_h\|^2_{H^{1}(K)}\bigg)^{\frac{1}{2}} \nonumber \\
&\lesssim h^{1+\epsilon}\|\b{u}\|_{\b{H^{{2+\epsilon}}}(\Omega)}\norm{\b{u-u_h}}_{\b{h}} \nonumber \\
&\leq \frac{1}{2C_1}h^{2(1+\epsilon)}\|\b{u}\|^2_{\b{H^{2+\epsilon}}(\Omega)} + \frac{C_1}{2}\norm{\b{u-u_h}}^2_{\b{h}} \nonumber\\
&\lesssim \frac{1}{2C_1}h^{2(1+\epsilon)}\|\b{u}\|^2_{\b{H^{2+\epsilon}}(\Omega)} +{C_1}h^{2(1+\epsilon)}\|\b{u}\|^2_{\b{H^{2+\epsilon}}(\Omega)}+ C_1\norm{\b{I_hu-u_h}}^2_{\b{h}},
\end{align}
for some infinitesimally small constant $0<C_1<\frac{\alpha}{2}$.  Next,  we consider the second integral in the right hand side of \eqref{17} as follows
\begin{align}\label{18}
\sum\limits_{e\in \mathcal{E}^C_h} \int_e {\tilde{\pi}_h(\sigma_n(\b{u}))}(u^1-u^1_h)~ds = \sum\limits_{e\in \mathcal{E}^C_h} \int_e {\tilde{\pi}_h(\sigma_n(\b{u}))}u^1~ds - \sum\limits_{e\in \mathcal{E}^C_h} \int_e {\tilde{\pi}_h(\sigma_n(\b{u}))}u^1_h~ds.
\end{align}
Further,  using the fact that $\tilde{\pi}_h(\sigma_n(\b{u}))$ is constant on each edge $e \in \mathcal{E}^C_h$ and $\sigma_n(\b{u}) \in C(\Gamma_C) \leq 0 $,  we have $\tilde{\pi}_h(\sigma_n(\b{u})) \leq 0$ on each edge  $e \in \mathcal{E}^C_h$.  Further,  exploiting the condition $\int_e u^1_h~ds \leq 0~\forall~e~\in~\mathcal{E}^C_h$,  we find $$\underset{e\in \mathcal{E}^C_h}{\sum} \int_e {\tilde{\pi}_h(\sigma_n(\b{u}))}u^1_h~ds \geq 0.$$
Thus,  \eqref{18} reduces to 
\begin{align}\label{19}
\sum\limits_{e\in \mathcal{E}^C_h} \int_e {\tilde{\pi}_h(\sigma_n(\b{u}))}(u^1-u^1_h)~ds \leq  \sum\limits_{e\in \mathcal{E}^C_h} \int_e {\tilde{\pi}_h(\sigma_n(\b{u}))}u^1~ds.
\end{align}
Using the fact that $\sigma_n(\b{u})u_n = 0~\text{on}~\Gamma_C$,  the right hand side of \eqref{19} changes to
\begin{align}\label{20}
\sum\limits_{e\in \mathcal{E}^C_h} \int_e {\tilde{\pi}_h(\sigma_n(\b{u}))}(u^1)~ds &= \sum\limits_{e\in \mathcal{E}^C_h} \int_e {(\tilde{\pi}_h(\sigma_n(\b{u})) - \sigma_n(\b{u})})u^1~ds \nonumber \\
&= \sum\limits_{e\in \mathcal{E}^C_h} \int_e {(\tilde{\pi}_h(\sigma_n(\b{u})) - \sigma_n(\b{u})})(u^1 - \tilde{\pi}_h(u^1))~ds .
\end{align}
On the similar lines as in \eqref{12},  we restrict the summation in \eqref{20} to the set $\mathcal{J}_2$ where  $$\mathcal{J}_2=  \{ e \in \mathcal{E}^C_h : u^1~ \text{vanishes atleast at one point on e} \in \mathcal{E}^C_h \}.$$
Finally,  using Lemma \ref{approx_result} we have 
\begin{align}\label{23}
\sum\limits_{e\in \mathcal{E}^C_h} \int_e {\tilde{\pi}_h(\sigma_n(\b{u}))}(u^1)~ds &= \sum\limits_{e\in \mathcal{J}_2} \int_e {(\tilde{\pi}_h(\sigma_n(\b{u})) - \sigma_n(\b{u})})(u^1 - \tilde{\pi}_h(u^1))~ds \nonumber \\
& \leq \sum\limits_{e\in \mathcal{J}_2} \| {\tilde{\pi}_h(\sigma_n(\b{u})) - \sigma_n(\b{u})}\|_{L^2(e)}\|u^1 - \tilde{\pi}_h(u^1)\|_{L^2(e)} \nonumber\\
& \lesssim \sum\limits_{e\in \mathcal{J}_2} h^{\frac{1}{2}+\epsilon} \| \sigma_n(\b{u})\|_{H^{\frac{1}{2}+\epsilon}(e)}h|u^1|_{H^1(e)}.
\end{align}
We have $u^1 \in H^{\frac{3}{2}+\epsilon}(\Gamma_C)$.  A use of Sobolev embedding theorem yields $u^1\in \mathcal{C}^{1,\epsilon}(\Gamma_C)$.  Since $u^1| _{\Gamma_C} \leq 0$ and $u^1$ vanishes at atleast one point on $e\in \mathcal{J}_2$ say $c$,  thus we have $u^1$  has a maxima at the point $c$.  Consequently,  $ {{du^1(c)}} =0$ where $du^1$ refers to the tangential derivative of $u^1$ along the edge $e$.  Utilising Lemma \ref{discrete1},  we find
\begin{align}\label{21}
|u^1|_{H^1(e)} = \|du^1\|_{L^2(e)} & \leq h^{\frac{1}{2}+\epsilon}\|du^1\|_{H^{\frac{1}{2}+\epsilon}(e)}\leq h^{\frac{1}{2}+\epsilon}\|u^1\|_{H^{\frac{3}{2}+\epsilon}(e)}.
\end{align}
Using \eqref{21} in \eqref{23},  we find
\begin{align}\label{new1}
\sum\limits_{e\in \mathcal{E}^C_h} \int_e {\tilde{\pi}_h(\sigma_n(\b{u}))}u^1~ds &\leq \sum\limits_{e\in \mathcal{J}_2} h^{\frac{1}{2}+\epsilon} \| \sigma_n(\b{u})\|_{{H^{\frac{1}{2}+\epsilon}(e)}}h^{\frac{3}{2}+\epsilon}\|u^1\|_{H^{\frac{3}{2}+\epsilon}(e)} \nonumber \\
& \leq h^{2(1+\epsilon)}  \bigg(\sum\limits_{e\in \mathcal{J}_2}\|\sigma_n(\b{u})\|^2_{{H^{\frac{1}{2}+\epsilon}(e)}}\bigg)^{\frac{1}{2}} \bigg(\sum\limits_{e\in \mathcal{J}_2} \|u^1\|^2_{H^{\frac{3}{2}+\epsilon}(e)} \bigg)^{\frac{1}{2}} \nonumber \\
& \leq h^{2(1+\epsilon)} \|\sigma_n(\b{u})\|_{H^{\frac{1}{2}+\epsilon}(\Gamma_C)}\|\b{u}\|_{\b{H^{\frac{3}{2}+\epsilon}}(\Gamma_C)} \nonumber\\
& \leq h^{2(1+\epsilon)} \|\b{u}\|_{\b{H^{\b{2+\epsilon}}}(\Omega)}.
\end{align}
Thus,  combining the equations \eqref{17},  \eqref{new},  \eqref{19} and \eqref{new1},  we obtain
\begin{align}\label{22}
Q_2 \lesssim \frac{1}{2C_1}h^{2(1+\epsilon)}\|\b{u}\|^2_{\b{H^{2+\epsilon}}(\Omega)} +{C_1}h^{2(1+\epsilon)}\|\b{u}\|^2_{\b{H^{2+\epsilon}}(\Omega)}+ C_1\norm{\b{I_hu-u_h}}^2_{\b{h}}.
\end{align} 
Finally,  the result follows by combining \eqref{1},  \eqref{2},  \eqref{10},  \eqref{11}, \eqref{30} and  \eqref{22}.  
\end{proof}
\section{\textbf{Discrete Lagrange Multiplier}}
\noindent
In this section,  we define the discrete counterpart $\b{\lambda_h}$  of Lagrange multiplier $\b{\lambda}$ on a suitable space.  Further we will be deriving the sign property of the discrete Lagrange multiplier which will be crucial in proving the \textit{a posteriori} error estimates. 
\par
\noindent
To this end,  we define an operator $\b{\beta_h}: \b{{V}_h} \longrightarrow \b{{M}_h}(\Gamma_C)$ taking the help of  projection operator $\b{\pi_h}$ as follows
\begin{align}\label{32}
\b{\beta_h}(\b{v_h}):=\b{\pi_h}(\b{\gamma_h}(\b{v_h})), \quad \forall \b{v_h} \in \b{{V}_h}.
\end{align}
Here in equation \eqref{32},  $\b{\gamma_h}: \b{{V}_h} \longrightarrow \b{{W}_h}$ refers to the trace map where
$$\b{{W}_h} := \big\{ \b{v_h} \in \b{L^2}(\Gamma_C) : \b{v_h}|_e \in [P_2(e)]^2~ \forall~e\in \mathcal{E}^C_h~\big\}.$$
For the ease of presentation,  we will be using the following representation of the map $\b{\beta_h}$ given by
\begin{align*}
\b{\beta_h}(\b{v_h}):= (\beta^1_h(\b{v_h}),\beta^2_h(\b{v_h}))~\quad \forall~\b{v_h}= (v^1_h, v^2_h) \in \b{{V}_h},
\end{align*}
where $\beta^1_h(\b{v_h})|_e= \frac{1}{h_e}\int \limits_{e} v_h^1~ds$ and $\beta^2_h(\b{v_h})|_e= \frac{1}{h_e}\int \limits_e v_h^2~ds ~\forall~e\in~\mathcal{E}^C_h.$ 
\vspace{0.4 cm}
\par
\noindent
Next,  we  establish some key properties of the operator $\b{\beta_h}$.
\begin{lemma}\label{onto}
The map $\b{\beta_h} : \b{{V}_h} \longrightarrow \b{{M}_h}(\Gamma_C)$ defined in \eqref{32} is onto.
\end{lemma}
\begin{proof}

For any $\b{w_h} \in \b{{M}_h}(\Gamma_C)$,  we have $\b{w_h}|_e \in [P_0(e)]^2 ~\forall~ e \in \mathcal{E}^C_h$.  {Concisely,  let $e_i,  1 \leq i \leq k$ be the enumeration of edges on $\Gamma_C$. To this end,  we define $\b{w_h}|_{e_i} =(c^1_i, c^2_i)$ for each edge $e_i \in \mathcal{E}^C_h$. }  Keeping in mind the presumption (\textbf{B}),  for each $e_i \in \mathcal{E}^C_h, 1 \leq i \leq k$  we identify the triangles $K_i \in \mathcal{T}_h$ such that $e_i \in \partial K_i$.  Accordingly,  we define $\b{v_h} \in \b{{V}_h}$ as
\begin{equation}\label{ch:five:sec:5:eq4:1}
\b{v_h} =  \begin{cases}
        (c^1_1,  c^2_1)~& \text{in } ~K_1,  \nonumber\\
        (c^1_2,  c^2_2) ~& \text{in } ~K_2,\nonumber\\
        (c^1_3,  c^2_3)~& \text{in } ~K_3,\nonumber\\
        . ... \nonumber\\
        .... \nonumber\\
        (c^1_k,  c^2_k)~& \text{in } ~K_k,\\
        (0, 0), ~&\text{otherwise.} ~~~~~~~~~~~~~~~~~~~~~~~~~
        \end{cases}
\end{equation}

It can be observed that $\b{\beta_h}(\b{v_h}) = \b{\pi_h}(\b{\gamma_h}(\b{v_h})) = \b{w_h}$.  Thus,  the map $\b{\beta_h}$ is onto.
\end{proof}
\begin{remark}
The surjective map $\b{\beta_h}: \b{{V}_h} \longrightarrow \b{{M}_h}(\Gamma_C)$ ensures a continuous right inverse $\b{\beta^{-1}_h}: \b{{M}_h}(\Gamma_C) \longrightarrow \b{{V}_h}$ defined by $\b{\beta_h^{-1
}}(\b{w_h}) = \b{v_h}$,   where  $\b{v_h}$ is defined in  Lemma \ref{onto}.
\end{remark}
\par
\noindent
With the assistance of map $\b{\beta_h}$,  we define discrete Lagrange multiplier $\b{\lambda_h}\in \b{{M}_h}(\Gamma_C)$  as
\begin{align}\label{33}
\b{(\b{\lambda_h},\b{w_h})_{L^2(\Gamma_C)}} := {L}(\b{\beta_h^{-1}}\b{w_h}) - \mathcal{A}_h(\b{u_h}, \b{\beta^{-1}_h}\b{w_h})~~\quad\forall~\b{w_h}\in \b{{M}_h}(\Gamma_C).
\end{align}
\begin{remark}
Since $\b{\lambda_h} \in \b{L^2}(\Gamma_C)$,  using  Cauchy-Schwarz  inequality,  it can be viewed as a functional on $\b{H^{\frac{1}{2}}}(\Gamma_C)$ as follows
\begin{align*}
\langle\b{ \lambda_h},  \b{w} \rangle_c := \int_{\Gamma_C} \b{ \lambda_h} \cdot \b{w}~\textbf{ds} ~\forall~\b{w}\in \b{H^{\frac{1}{2}}}(\Gamma_C).
\end{align*}
\end{remark} 
\noindent
In order to carry out further analysis,  we work with an additional functional space $\b{{V}_Q}$ defined as
\begin{align}\label{Vq space}
\b{{V}_Q}=\bigg\{\b{v_h} \in \b{{V}_h}:~{\beta^1_h(\b{v_h)}=0}\bigg\}.
\end{align} 
\begin{remark}
It can be observed that $\b{v_h} = \b{u_h} \pm \b{v} \in \b{{K}_h}$ for $\b{v} \in \b{{V}_Q}$.  Thus,  a use of \eqref{discrete} yields
\begin{align}\label{equality11}
\mathcal{A}_h(\b{u_h}, \b{v}):={L}(\b{v})~\forall~\b{v}~\in~\b{{V}_Q}. 
\end{align}
\end{remark}
\noindent
The choice of the map $\b{\lambda_h}$ plays a key role in establishing further estimates.  In the upcoming lemma,  we will show the well definedness of the map $\b{\lambda_h}$ taking the help of the functional space $\b{{V}_Q}$.
\begin{lemma}
The map $\b{\lambda_h} \in \b{{M}_h(\Gamma_C)}$ defined in \eqref{33} is well defined.
\end{lemma}
\begin{proof}
Let $\b{w_h}$,  $\b{q_h} \in \b{{M}_h}(\b{\Gamma_C)}$ be such that $\b{w_h}=\b{q_h}$.  Then, there exist $\b{\tilde{w_h}}$, $\b{\tilde{q_h}} \in \b{{V}_h}$ such that $\b{\beta_h(\tilde{w_h})}=\b{w_h}$ with $\b{\tilde{w_h}} = \b{\beta_h^{-1}}(\b{w_h}) $ and $\b{\beta_h(\tilde{q_h})}=\b{q_h}$ with $\b{\tilde{q_h}} = \b{\beta_h^{-1}}(\b{q_h}) $.  Thus, $\b{\beta_h}(\tilde{\b{w_h}}-\tilde{\b{q_h}})=\b{0}$ which implies $\tilde{\b{w_h}}-\tilde{\b{q_h}} \in \b{{V}_Q}$ using $\eqref{Vq space}$.   Further,  a use of equation $\eqref{equality11}$ yields
\begin{align*}
&\mathcal{A}_h(\b{u_h},  \tilde{\b{w_h}}-\tilde{\b{q_h}}) - {L}(\tilde{\b{w_h}}-\tilde{\b{q_h}})=\b{0}. 
\end{align*}
Thus,  we have $\b{(\b{\lambda_h},\b{w_h})_{L^2(\Gamma_C)}} = \b{(\b{\lambda_h},\b{q_h})_{L^2(\Gamma_C)}}$.
\end{proof}
\begin{lemma}\label{eqn2}
For any $\b{v_h} \in \b{{V}_h}$,  the following results  hold
\begin{align}\label{eqn1}
(\b{\lambda_h}, \b{v_h})_{\b{L^2(\Gamma_C)}} = L(\b{v_h}) - \mathcal{A}_h(\b{u_h}, \b{v_h}).
\end{align}
\end{lemma}
\begin{proof}
The construction of map $\b{\beta_h}$ ensures that 
\begin{align}\label{equality}
\int \limits_e\big(\b{\beta_h}(\b{v_h}) -\b{v_h}\big)~\textbf{ds} =\b{0}~\forall~e~\in~\mathcal{E}^C_h.
\end{align}
We have, 
\begin{align}\label{eqn}
(\b{\lambda_h}, \b{v_h})_{\b{L^2(\Gamma_C)}} &= (\b{\lambda_h}, \b{\beta_h(v_h)})_{\b{L^2(\Gamma_C)}} \nonumber\\
&= L\big((\b{\beta^{-1}_h\circ\b{\beta_h}})(\b{v_h})\big)-\mathcal{A}_h(\b{u_h}, ( \b{\beta^{-1}_h}\circ\b{\beta_h})(\b{v_h})) \nonumber \\
&= L\big((\b{\beta^{-1}_h\circ\b{\beta_h)}}(\b{v_h}))-\mathcal{A}_h(\b{u_h},  (\b{\beta^{-1}_h}\circ\b{\beta_h})(\b{v_h})) -L(\b{v_h}) + \mathcal{A}_h(\b{u_h}, \b{v_h})\nonumber \\&~~~+L(\b{v_h}) - \mathcal{A}_h(\b{u_h}, \b{v_h}) \nonumber \\
&= L\big( (\b{\beta^{-1}_h\circ\b{\beta_h}})(\b{v_h}) - \b{v_h}\big) - \mathcal{A}_h(\b{u_h},  (\b{\beta^{-1}_h}\circ\b{\beta_h})(\b{v_h})-\b{v_h})+L(\b{v_h}) - \mathcal{A}_h(\b{u_h}, \b{v_h}).
\end{align}
Since $(\b{\beta^{-1}_h}\circ\b{\beta_h})(\b{v_h})-\b{v_h}\in \b{{V}_Q}$,  therefore using \eqref{equality11},  we  conclude \eqref{eqn1}.
\end{proof}
\par
\noindent
Further,  in view of Lemma \ref{eqn2},  for $\b{v_h}=(v^1_h, v^2_h) \in \b{{V}_h}$,  we give another representation of $(\b{\lambda_h}, \b{\beta_h}(\b{v_h}))_{\b{L^2}(\Gamma_C)} $ as 
\begin{align}\label{represent}
(\b{\lambda_h}, \b{\beta_h(v_h}))_{\b{L^2}(\Gamma_C)} = ({\lambda^1_h},  \beta^1_h(\b{v_h}))_{{L^2}(\Gamma_C)}
+ ({\lambda^2_h}, \beta^2_h(\b{v_h}))_{{L^2}(\Gamma_C)},
\end{align}
where,
\begin{align*}
({\lambda^1_h},  \beta^1_h(\b{v_h}))_{{L^2(\Gamma_C)}}&:= L((v^1_h,0)) -\mathcal{A}_h(\b{u_h}, (v^1_h,0)),\\ 
({\lambda^2_h},  \beta^2_h(\b{v_h}))_{{L^2(\Gamma_C)}}&:= L((0,  v^2_h)) -\mathcal{A}_h(\b{u_h}, (0, v^2_h)).
\end{align*}
In the next lemma,  we derive the sign property of discrete Lagrange multiplier $\b{\lambda_h}=(\lambda^1_h, \lambda^2_h)$. 
\begin{lemma}\label{sign3}
It holds that 
\begin{align*}
\begin{split}
\begin{aligned}
\lambda^1_h|_{e} &\geq 0\\
\lambda^2_h|_{e} &= 0 
\end{aligned} \quad \forall~e\in \mathcal{E}^C_h . 
\end{split}
\end{align*}
\begin{proof}
The proof of the lemma can be accomplished by  constructing a suitable test function $\b{v_h}\in \b{{V}_h}$.  To this end,   choose any arbitrary edge $e \in \mathcal{E}^C_h$. Let $p \in \mathcal{V}_e \cup \mathcal{M}_e$ be an arbitrary node.
Construct $\b{v_h} \in \b{{V}_h}$ such that $\b{v_h}=0~ \text{on}~ \Omega\setminus \bar{e}$ and 
\begin{align*}
\begin{split}
\begin{aligned}
\b{v_h}(z) = \begin{cases} &(1,0)~~~\text{if}~ z=p, \\
&(0,~0)~~~\text{if}~z\in (\mathcal{V}_e \cup \mathcal{M}_e)\setminus p. \end{cases}
\end{aligned}
\end{split}
\end{align*}
It can be verified that $\b{u_h}-\b{v_h}\in \b{{K}_h}$. Thus,  using the discrete variational inequality \eqref{discrete},  we find
\begin{align}\label{eqn3}
\mathcal{A}_h(\b{u_h}, \b{v_h}) \leq {L}(\b{v_h}).
\end{align}
which implies $(\b{\lambda_h}, \b{\b{v_h}})_{\b{L^2}(\Gamma_C)}\geq \b{0}$.  On one hand,  we have
\begin{align} \label{lam11}
(\b{\lambda_h}, \b{\beta_h(\b{v_h}}))_{\b{L^2}(\Gamma_C)} &=(\lambda^1_h,  \beta^1_h(\b{v_h}))_{L^2(\Gamma_C)} +  (\lambda^2_h,  \beta^2_h(\b{v_h}))_{L^2(\Gamma_C)} \nonumber \\
&= \int_e\lambda^1_h  \beta^1_h(\b{v_h})~ds + \int_e\lambda^2_h  \beta^2_h(\b{v_h})~ds\nonumber \\
&=  \int_e\lambda^1_h  \beta^1_h(\b{v_h})~ds, 
\end{align}
thus,  in view of equation \eqref{equality}, Lemma \ref{eqn2},  equation \eqref{eqn3} together with equation \eqref{lam11} we find
\begin{align}\label{Eqn4}
(\b{\lambda_h}, \b{\beta_h(\b{v_h}}))_{\b{L^2}(\Gamma_C)} = \int\limits_e\lambda^1_h  \beta^1_h(\b{v_h})~ds \geq 0.
\end{align}
Taking into the account that  the quantity $\lambda^1_h \beta^1_h(\b{v_h}) $ is constant on edge $e$ and $ \beta^1_h(\b{v_h}) =\frac{1}{h_e} \int \limits_e v_h^1~ds > 0$,  we find $\lambda^1_h|_e > 0$.  On the similar lines,  we derive that  $\lambda^2_h|_e = 0$.  For that,  we work with the test function $\b{v_h} \in \b{ {V}_h}$ as
\begin{align*}
\begin{split}
\begin{aligned}
\b{v_h}(z) = \begin{cases} &(0, 1) ~~~\text{if}~ z=p, \\
&(0,0)~~~~\text{if}~z\in (\mathcal{V}_e \cup \mathcal{M}_e)\setminus p. \end{cases}
\end{aligned}
\end{split}
\end{align*}
Observe that $\b{u_h}\pm\b{v_h}\in \b{{K}_h}$ which in turn shows that
\begin{align}\label{equalll}
\mathcal{A}_h(\b{u_h}, \b{v_h}) = {L}(\b{v_h})
\end{align}
using discrete variational inequality \eqref{discrete}.  Thus, Lemma \ref{eqn2} assures $(\b{\lambda_h}, \b{\beta_h(\b{v_h}}))_{\b{L^2}(\Gamma_C)} = \b{0}$.  Also,    the construction of test function $\b{v_h}$ gives $$(\b{\lambda_h}, \b{\beta_h(\b{v_h}}))_{\b{L^2}(\Gamma_C)} =  \int_e\lambda^2_h  \beta^2_h(\b{v_h})~ds =   \lambda^2_h|_e \int \limits_e v^2_h~ds.$$
Thus,  we have $\lambda^2_h|_e=0$ since $ \int \limits_e v_h^2~ds > 0$.  This completes the proof of this lemma.
\end{proof}
\end{lemma}
\par
\noindent
Next,  we decompose our contact edges in two sets as  the discrete contact set $\b{\mathcal{C}_h}$ and discrete non-contact set  $\b{\mathcal{N}_h}$ as follows
\begin{align*}
\b{\mathcal{C}_h} =\bigg\{ e \in \mathcal{E}_h^C:  ~~\int_eu^1_h~ds =0\bigg\} ,\\
\b{\mathcal{N}_h} =\bigg\{ e \in \mathcal{E}_h^C:  ~~\int_eu^1_h~ds < 0\bigg\}. 
\end{align*}
\begin{remark}\label{rem7}
It is worth mentioning that $\lambda^1_h|_e = 0 ~\forall~e\in ~\b{\mathcal{N}_h}$ as for the edge $e \in \b{\mathcal{N}_h}$. To realize this,  let $p \in \mathcal{V}_e \cup \mathcal{M}_e$ be an arbitrary node.  For sufficiently small $\delta>0$,  define  
\begin{align*}
\b{v_h} = \b{u_h} \pm \delta\phi_p\b{e_1}, 
\end{align*}
where $\phi_p$ refers to the $P_2$ Lagrange basis function corresponding to the node $p$.  Observe that $\b{u_h} \pm \b{v_h} \in \b{{K}_h}$. Thus,  $(\b{\lambda_h}, \b{\beta_h(\b{v_h}}))_{\b{L^2}(\Gamma_C)} = L(\b{v_h}) - \mathcal{A}_h(\b{u_h}, \b{v_h}) = 0$.  Using the similar arguments used to prove Lemma \ref{represent},  we note that $\lambda^1_h|_e = 0 ~\forall~e~\in \b{\mathcal{N}_h}$. 
\end{remark}
\subsection{Enriching Map}
\noindent
In the \textit{a posteriori} error analysis of non-conforming methods and discontinuous Galerkin methods, enriching map plays a vital role as it is continuous approximation of discontinuous functions  \cite{Brenner::2004:Korn, Brenner:2001:Poincare}.  In the present context,  we construct enriching map $\b{E_h} :  \b{{V}_h} \longrightarrow \b{{V}_c}$ using the technique of averaging as follows. Let $ \b{v_h} \in \b{{V}_h} $,
\begin{itemize}
\item For $p \in \mathcal{V}^D_h  \cup \mathcal{M}^D_h$,  we define  $\b{E_h} \b{v_h}(p) = 0$.  
\vspace{0.2 cm}
\item For $p \in (\mathcal{V}_h  \cup \mathcal{M}_h) \backslash (\mathcal{V}_h^D  \cup \mathcal{M}_h^D)$
\begin{align*}
\b{E_h} \b{v_h}(p) = \dfrac{1}{|\mathcal{T}_p|} \sum\limits_{K \in \mathcal{T}_p} \b{v_h}|_K(p). 
\end{align*}
\end{itemize}

\vspace{0.3 cm}
\par
\noindent
Next,  let us recollect approximation properties of enriching map which can be proved using the scaling arguments \cite{TG:2016:VIDG1}.
\begin{lemma}\label{enriching}
Let $\b{v_h} \in \b{{V}_h}$.  It hold that
\begin{align*}
\sum\limits_{K \in \mathcal{T}_h} h^{-2}_K \|\b{E_h}\b{v_h} - \b{v_h}\|^2_{\b{L^2}(K)} &\lesssim  \sum\limits_{e \in \mathcal{E}^0_h} \frac{1}{h_e}\| \sjump{\b{v_h}} \|^2_{\b{L^2}(e)}, \\
\sum\limits_{K \in \mathcal{T}_h} \|\nabla(\b{E_h}\b{v_h} - \b{v_h})\|^2_{\b{L^2}(K)} &\lesssim  \sum\limits_{e \in \mathcal{E}^0_h} \frac{1}{h_e}\| \sjump{\b{v_h}} \|^2_{\b{L^2}(e)}.
\end{align*} 
\end{lemma} 
\noindent
\section{\textbf{A Posteriori Error Estimates}}
\noindent
In this section,  we perform \textit{a posteriori} error estimation wherein we introduce residual based \textit{a posteriori} error estimator and establish the reliability and efficiency of the estimator.  For this,  we define a Galerkin functional which help us to measure error in both unknowns $\b{u_h}$ and $\b{\lambda_h}$ in certain norms followed by that we derive a global upper bound for error term by estimator.   Finally we conclude the section with a discussion on the efficiency results.
\par
\vspace{0.3 cm}
\noindent
Define the following a posteriori estimator contributions
\begin{align*}
\eta^2_1 &:= \sum\limits_{K\in \mathcal{T}_h}h^2_K \|\b{f} + \b{div} \b{\sigma}_h(\b{u_h})\|^2_{\b{L^2}(K)}, \\
\eta^2_2 &:= \sum_{e\in \mathcal{E}^i_h}h_e \|\sjump{ \b{\sigma}_h(\b{u_h})}\|^2_{\b{L^2}(e)}, \\
\eta^2_3 &:= \sum_{e\in \mathcal{E}^N_h}h_e \|\b{g} - \b{\sigma}_h(\b{u_h})\b{n}\|^2_{\b{L^2}(e)}, \\
\eta^2_4 &:= \sum_{e\in \mathcal{E}^C_h}h_e \|\b{\lambda_h} + \b{\sigma}_h(\b{u_h})\b{n}\|^2_{\b{L^2}(e)}, \\
\eta^2_5 &:= \sum_{e\in \mathcal{E}^0_h}\frac{1}{h_e} \|\sjump{\b{u_h}}\|^2_{\b{L^2}(e)}.
\end{align*} 
The total residual error estimator is given by 
\begin{align*} 
\eta^2_h := \eta^2_1 + \eta^2_2 + \eta^2_3 +\eta^2_4 +\eta^2_5 - \sum_{e\in \b{\mathcal{C}_h}} \int_{e}  \lambda^1_h (E^1_h\b{u_h})^{-}~ds + \sum_{e\in \mathcal{E}^C_h} \|\gamma_c(E^1_h\b{u_h})^{+}\|^2_{H^{\frac{1}{2}}(e)}.
\end{align*}
Let us introduce a new norm on the space $\b{{V}_h} $ as
\begin{align*}
\norm{\b{v_h}} := |\b{v_h}|^2_{\b{h}} + |\b{v_h}|^2_{\b{*}},
\end{align*}
where the quantities $|\b{v_h}|_{\b{h}} $  and $|\b{v_h}|_{\b{*}} $ are defined in equation \eqref{relation}.  
\vspace{0.1 cm}
\par
\noindent
\subsection{\textbf{Reliable A Posteriori Error Estimates}}
This subsection is intended to establish the reliability of residual error estimator wherein the main result is stated in the following theorem.
\begin{theorem}\label{main}
Let $\b{u}$ be the solution of continuous formulation \eqref{weak1} and $\b{u_h}$ be the solution of discrete variational inequality \eqref{discrete},  then we have the following bound 
\begin{align*}
{\norm{\b{u-u_h}}}^2 + \|\b{\lambda}-\b{\lambda_h}\|^2_{\b{H^{-\frac{1}{2}}}(\Gamma_C)}&\lesssim  \eta^2_h.
\end{align*}
\end{theorem}
\par
\noindent
To prove this theorem, we require an intermediate result.
Analogous to the case of linear elliptic problem,  we define Galerkin functional which plays an essential role  in deriving the upper and lower bound of total residual error estimator $\eta_h$.  To this end,  we
define a map $\b{\mathcal{G}_h}: \b{{V}} \longrightarrow \mathbb{R}$ as
\begin{align*}
\b{\mathcal{G}_h}(\b{v}) := a_h(\b{u-u_h}, \b{v}) +\langle \b{\lambda} -\b{\lambda_h}, \b{\gamma_c}(\b{v}) \rangle_c ~\forall~\b{v}\in \b{{V}}. 
\end{align*}
In the next lemma,  we will see the the bound on error in displacement term and Lagrange multiplier by dual norm of functional $\b{\mathcal{G}_h}$ and duality pairing between Lagrange multiplier and displacements.
\begin{lemma}\label{lem}
It holds that 
\begin{align*}
{\norm{\b{u-u_h}}}^2 + \|\b{\lambda}-\b{\lambda_h}\|^2_{\b{H^{-\frac{1}{2}}}(\Gamma_C)} &\lesssim  \|\b{\mathcal{G}_h}\|^2_{\b{H^{-1}}(\O)} + \sum_{e\in \mathcal{E}^0_h}\frac{1}{h_e} \|\sjump{\b{u_h}}\|^2_{\b{L^2}(e)} \\&- \langle \b{\lambda_h} -\b{\lambda}, \b{\gamma_c(E_h\b{u_h})} -\b{\gamma_c(u)} \rangle_c.
\end{align*}
\begin{proof}
This lemma can be proved on similar lines as in Lemma 5.1 of article \cite{TG:2016:VIDG1},  therefore proof is  omitted.
\end{proof}
\end{lemma}
\vspace{0.2 cm}
\par
\noindent
With the help of Lemma \ref{lem},  we establish the reliability of \textit{a posteriori} error estimator.
\begin{proof}[Proof of Theorem \ref{main}]
To start with,  we choose $\b{\zeta} \in \b{{V}}$  to be arbitrary and corresponding to this,  let $\b{\zeta_h} \in \b{{V}_c}$ be the approximation of $\b{\zeta}$ satisfying the estimates in Lemma \ref{clement}.
We have,  
\begin{align}\label{400}
\b{\mathcal{G}_h}(\b{\zeta}) =\b{\mathcal{G}_h}(\b{\zeta} -\b{\zeta_h}) + \b{\mathcal{G}_h}(\b{\zeta_h}).
\end{align}
Exploiting the definition of Galerkin functional $\b{\mathcal{G}_h}$,  we  bound the first term in the the RHS of \eqref{400} as follows
\begin{align*}
\b{\mathcal{G}_h}(\b{\zeta} -\b{\zeta_h}) &= a_h(\b{u - u_h},  \b{\zeta} -\b{\zeta_h}) + \langle \b{\lambda} -\b{\lambda_h}, \b{\gamma_c(\zeta)} -\b{\gamma_c(\zeta_h)}
 \rangle_c\\
&= L(\b{\zeta -\zeta_h })- a_h(\b{u_h},   \b{\zeta}- \b{\zeta_h})-\langle \b{\lambda_h},  \b{\gamma_c(\zeta)}- \b{\gamma_c(\zeta_h)} \rangle_c 
\end{align*} 
where the latter equation followed from \eqref{lagrange1}.  Further,  a use of integration by parts,  Cauchy-Schwarz inequality,  discrete trace inequality together with the Cl$\acute{e}$ment approximation properties (Lemma \ref{clement}),  yields
\begin{align}\label{499}
\b{\mathcal{G}_h}(\b{\zeta} -\b{\zeta_h}) &= \sum_{K\in \mathcal{T}_h} \int\limits_K \b{f} \cdot (\b{\zeta} - \b{\zeta_h})~\textbf{dx} + \sum_{e\in \mathcal{E}^N_h} \int\limits_e \b{g} \cdot (\b{\zeta} - \b{\zeta_h})~\textbf{ds}-\sum_{K\in \mathcal{T}_h} \int\limits_K \b{\sigma}_h(\b{u_h}): \b{\varepsilon}_h(\b{\zeta} - \b{\zeta_h})~\textbf{dx} \\&- \sum_{e\in \mathcal{E}^C_h} \int\limits_e \b{\lambda_h} \cdot (\b{\zeta} - \b{\zeta_h})~\textbf{ds} \nonumber\\
&= \sum_{K\in \mathcal{T}_h} \int\limits_K \big(\b{f} + \b{div}\b{\sigma}_h(\b{u_h})\big) \cdot (\b{\zeta} - \b{\zeta_h})~\textbf{dx} +  \sum_{e\in \mathcal{E}^N_h} \int\limits_e \big(\b{g} - \b{\sigma}_h(\b{u_h})\b{n}\big)  \cdot (\b{\zeta} - \b{\zeta_h})~\textbf{ds} \nonumber\\
&~~- \sum_{e\in \mathcal{E}^C_h} \int\limits_e \big(\b{\lambda_h} + \b{\sigma}_h(\b{u_h})\b{n}\big)  \cdot (\b{\zeta} - \b{\zeta_h})~\textbf{ds} -\sum_{e\in \mathcal{E}^i_h} \int\limits_e \sjump {\b{\sigma}_h(\b{u_h})}\cdot (\b{\zeta} - \b{\zeta_h})~\textbf{ds} \nonumber\\
&\lesssim (\eta_1 +\eta_2 +\eta_3 +\eta_4)|\b{\zeta}|_{\b{H^1}(\O)}.
\end{align}
Further,  we employ the equation $\eqref{eqn1}$ together with $\eqref{reformulate}$ and $\eqref{bound3}$ in order to bound second term of $\eqref{400}$ as follows
\begin{align} \label{500}
\b{\mathcal{G}_h}(\b{\zeta_h}) &= L(\b{\zeta_h}) - a_h( \b{u_h},  \b{\zeta_h}) - \langle \b{\lambda_h},  \b{\gamma_c}(\b{\zeta_h}) \rangle_c  \nonumber\\
&= ( \b{\lambda_h},  \b{\zeta_h} )_{\b{L^2}(\Gamma_C)} - \langle \b{\lambda_h},  \b{\gamma_c}(\b{\zeta_h}) \rangle_c - b_h( \b{u_h}, \b{ \zeta_h}) \nonumber \\
&\lesssim  \eta_5|\b{\zeta}|_{\b{H^1}(\O)}.
\end{align}
Thus,  combining \eqref{499} and \eqref{500} and using discrete Cauchy-Schwarz inequality,  we obtain
\begin{align}\label{main1}
\|\b{\mathcal{G}_h}\|_{\b{H^{-1}}(\O)} \lesssim \bigg(\eta^2_1+\eta^2_2+\eta^2_3+\eta^2_4+\eta^2_5\bigg)^{\frac{1}{2}}.
\end{align}
Lastly,  we need to derive the upper bound of  $- \langle \b{\lambda_h} -\b{\lambda}, \b{\gamma_c(E_h\b{u_h})} -\b{\gamma_c(u)} \rangle_c$.  Taking into the account $\b{{K}_h} \nsubseteq \b{{K}}$,  we define $\b{u_{conf}}= (u^1_{conf}, u^2_{conf}) := min\{\b{E_h}\b{u_h}, \b{0}\} \in \b{{K}}$.  In view of Remark \ref{rem1} and Lemma \ref{sign3},  we have 
\begin{align*}
 \langle \b{\lambda_h} - \b{\lambda},  {\b{\gamma_c}(\b{E_h}\b{u_h})} -\b{\gamma_c(u)} \rangle_c &=\langle \lambda^1_h - \lambda^1,  \gamma_c(E^1_h\b{u_h}) -\gamma_c({u}^1) \rangle. 
\end{align*}
Further,  using the continuous variational inequality \eqref{weak1} and Young's inequality,   we have 
\begin{align}\label{312}
\langle \lambda^1_h - \lambda^1,  \gamma_c(E^1_h\b{u_h}) -\gamma_c({u}^1) \rangle &= \langle \lambda^1_h ,  \gamma_c(E_h^1\b{u_h}) - \gamma_c(u^1)  \rangle -  \langle \lambda^1 ,  \gamma_c(E_h^1\b{u_h}) -\gamma_c(u^1) \rangle \nonumber\\
&= \langle \lambda^1_h ,  \gamma_c(E_h^1\b{u_h})- \gamma_c(u^1) \rangle -  \langle \lambda^1 ,   \gamma_c(E_h^1\b{u_h}) - \gamma_c(u^1_{conf}) \rangle \nonumber\\&-  \langle \lambda^1 ,   \gamma_c(u^1_{conf}) - \gamma_c(u^1) \rangle \nonumber \\
& \geq \langle \lambda^1_h ,  \gamma_c(E_h^1\b{u_h})- \gamma_c(u^1) \rangle -  \langle \lambda^1 ,   \gamma_c(E_h^1\b{u_h}) -\gamma_c(u^1_{conf}) \rangle \nonumber \\
&= \langle \lambda^1_h ,  \gamma_c(u^1_{conf})  -\gamma_c( u^1)  \rangle -  \langle \lambda^1 - \lambda^1_h, ~  \gamma_c(E_h^1\b{u_h}) - \gamma_c(u^1_{conf}) \rangle \nonumber  \\
&= \langle \lambda^1_h , \gamma_c(u^1_{conf})  - \gamma_c(u^1)  \rangle - \theta \|\lambda^1 - \lambda^1_h\|^2_{H^{-\frac{1}{2}}(\Gamma_C)} \nonumber \\&- \frac{1}{\theta}\|\gamma_c(E_h^1\b{u_h}) - \gamma_c(u^1_{conf})\|^2_{H^{\frac{1}{2}}(\Gamma_C)},
\end{align}
for some arbitrary small $\theta>0$.  Note that on $\Gamma_C$,  we have
\begin{align} \label{313}
\gamma_c( u^1_{conf}) - \gamma_c(u^1)  \geq \gamma_c(u^1_{conf}).
\end{align}
Inserting \eqref{313} into \eqref{312} and further using the fact that $\lambda_h^1$ is non negative on $\Gamma_C$,  we find 
\begin{align}\label{314}
\langle \lambda^1_h - \lambda^1, \gamma_c(E^1_h\b{u_h}) -\gamma_c({u}^1) \rangle  &\geq  \langle \lambda^1_h,  \gamma_c(u^1_{conf}) \rangle- \theta \|\lambda^1 - \lambda^1_h\|^2_{H^{-\frac{1}{2}}(\Gamma_C)} \nonumber\\&- \frac{1}{\theta}\|\gamma_c(E_h^1\b{u_h}) - \gamma_c(u^1_{conf})\|^2_{H^{\frac{1}{2}}(\Gamma_C)}. 
\end{align}
Further, we  estimate the first term in right hand side of \eqref{314} using  Remark \ref{rem7} as follows
\begin{align}\label{315}
\langle \lambda^1_h ,  \gamma_c(u^1_{conf}) \rangle = \int_{\Gamma_C}  \lambda^1_h u^1_{conf} ~ds &= \sum_{e\in \b{\mathcal{C}_h}} \int_{e}  \lambda^1_h u^1_{conf} ~ds + \sum_{e\in \b{\mathcal{N}_h}} \int_{e}  \lambda^1_h u^1_{conf} ~ds \nonumber \\&=  \sum_{e\in \b{\mathcal{C}_h}} \int_{e}  \lambda^1_h u^1_{conf} ~ds \nonumber\\  &=\sum_{e\in \b{\mathcal{C}_h}} \int_{e}  \lambda^1_h (E^1_h\b{u_h})^{-}~ds.
\end{align}
Thus combining \eqref{314} and \eqref{315},  we find 
\begin{align}\label{main2}
-\langle \lambda^1_h - \lambda^1,  \gamma_c(E^1_h\b{u_h}) -\gamma_c({u}^1) \rangle  \leq &- \sum_{e\in \b{\mathcal{C}_h}} \int_{e}  \lambda^1_h (E^1_h\b{u_h})^{-} ~ds + \theta \|\lambda^1 - \lambda^1_h\|^2_{H^{-\frac{1}{2}}(\Gamma_C)} \notag \\&+\frac{1}{\theta} \|\gamma_c(E_h^1\b{u_h}) -\gamma_c (u^1_{conf})\|^2_{H^{\frac{1}{2}}(\Gamma_C)}.
\end{align}
wherein,  $\gamma_c(E_h^1\b{u_h}) -\gamma_c (u^1_{conf})=\gamma_c(E^1_h\b{u_h})^{+}.$
Finally,  using \eqref{main1} and \eqref{main2},  together with Lemma \ref{lem} we obtain
\begin{align}\label{main3}
{\norm{\b{u-u_h}}}^2 + \|\b{\lambda}-\b{\lambda_h}\|^2_{\b{H^{-\frac{1}{2}}}(\Gamma_C)}&\lesssim  \eta^2_h.
\end{align} 
\end{proof}
\par 
\noindent
Note,  for the ease of notation,  we set $\eta_6 := - \underset{e\in \b{\mathcal{C}_h}}{\sum} \int_{e}  \lambda^1_h (E^1_h\b{u_h})^{-}~ds$ and $\eta_7 :=  \sum\limits_{e\in \mathcal{E}^C_h} \|\gamma_c(E^1_h\b{u_h})^{+}\|^2_{H^{\frac{1}{2}}(e)}$.
\subsection{\textbf{Efficiency of A Posteriori Error Estimator}}
\par
In this subsection,  we discuss the local efficiency estimates for \textit{a posteriori} error control for the quadratic DG FEM.  Therein, the standard techniques based on bubble functions  can be used to prove the efficiency estimates { as discussed in articles \cite{TG:2016:VIDG1, KP:2022:QuadSignorini}, thus the proofs are omitted.} \ {The efficiency of the estimator contributions $\eta_6$ and $\eta_7$ is still not clear but they are taken into the consideration while performing the numerical experiments}.  We define the oscillation terms as follows
\begin{align*}
{Osc}(\b{f}, \O)^2 := \sum_{K \in \mathcal{T}_h} {Osc}(\b{f}, K)^2,  \\ {Osc}(\b{g}, \Gamma_N)^2 := \sum_{e \in \mathcal{E}^N_h} {Osc}(\b{g}, e)^2, 
\end{align*}
where,  ${Osc}(\b{f}, K):= h_K\underset{\b{\bar{f}}\in [P_1(K)]^2}{min}\|\b{f-\bar{f}}\|_{\b{L^2}(K)}$ and 
 ${Osc}(\b{g}, e):=h_e^{\frac{1}{2}}\underset{\b{\bar{g}}\in [P_1(e)]^2}{min}\|\b{g-\bar{g}}\|_{\b{L^2}(e)}$ for any $K \in \mathcal{T}_h$ and $e \in \mathcal{E}^N_h$,  respectively.
 
 \noindent
{Below we state the theorem  concerning the efficiency of the residual estimator $\eta_h$.}
\begin{theorem}
Let $\b{u}$ and $\b{u_h}$ be the solution of equation \eqref{weak1} and \eqref{discrete},  respectively.  Then,  the following hold
\begin{itemize}
\item \textbf{Bound for $\eta_1$} : $ \underset{K \in \mathcal{T}_h}{\sum}h_K^2 \|\b{f} + \b{div} \b{\sigma}_h(\b{u_h})\|^2_{\b{L^2}(K)} \lesssim \norm{\b{u-u_h}}^2 +{Osc}(\b{f}, \O)^2$,\\
\item  \textbf{Bound for $\eta_2$}  : $\underset{{e \in \mathcal{E}^i_h}}{\sum}h_e \|\sjump{\b{\sigma}_h(\b{u_h})}\|^2_{\b{L^2}(e)} \lesssim \norm{\b{u-u_h}}^2+{Osc}(\b{f}, \O)^2$,\\
 \item \textbf{Bound for $\eta_3$}: $ \underset{e \in \mathcal{E}^N_h}{\sum}h_e \|\b{\sigma}_h(\b{u_h})\b{n}-\b{g}\|^2_{\b{L^2}(e)} \lesssim \norm{\b{u-u_h}}^2 + {Osc}(\b{f}, \O)^2+{Osc}(\b{g}, \Gamma_N)^2$,\\
 \item  \textbf{Bound for $\eta_4$}: $\underset{e \in \mathcal{E}^C_h}{ \sum}h_e \|\b{\sigma}_h(\b{u_h})\b{n}+\b{\lambda_h}\|^2_{\b{L^2}(e)} \lesssim \norm{\b{u-u_h}}^2 + \|\b{\lambda} -\b{\lambda_h}\|^2_{\b{H^{-\frac{1}{2}}}(\Gamma_C)}+ {Osc}(\b{f}, \O)^2$\\
 \end{itemize}
 ~~~~~~~~~~~~~~~~~~~~~~~~~~~~~~~~~~~~~~~~~~~~~~~~~~~~~~~~~~~~~~~~~~~~~~~~~~~~~$+{Osc}(\b{g}, \Gamma_N)^2,$ \\
 \begin{itemize}
 \item  \textbf{Bound for $\eta_5$}: $  \underset{e\in \mathcal{E}^0_h}{\sum}\dfrac{1}{h_e} \|\sjump{\b{u_h}}\|^2_{\b{L^2}(e)} \lesssim \norm{\b{u-u_h}}^2$.
\end{itemize} 

\end{theorem}

\section{\textbf{Numerical Experiments}}
\noindent
The goal of this section is to examine two contact model problems in 2D that substantiate the theoretical results.  The implementation have been carried out  in MATLAB 2020B.  The numerical results are reported for two DG schemes namely SIPG and NIPG wherein the penalty parameter $\eta$ is chosen to be 70 and 70$\nu$, respectively for both problems.  We use primal dual active set strategy \cite{Wohlmuth:2005:PrimalDual} in order to compute discrete solution $\b{u_h}$.
\vspace{0.5 cm}
\par
\noindent
\underline{\textbf{Model Problem 1:}} \textit{(Contact with a rigid foundation)}
\vspace{0.2 cm}
\par
\noindent
In this model problem,  we simulate the deformation of linear elastic unit square represented by $\O =(0,1) \times (0,1)$ which comes in  contact with a rigid foundation at the bottom of unit square in the region $(0,1) \times \{0\}$.  Moreover,  the top of unit elastic square is clamped i.e. the Dirichlet boundary condition is imposed on $(0,1) \times\{1\}$.  The Neumann forces are acting on left and right side of the unit square which can be computed using the exact solution $\b{u}= (u_1, u_2) = (y^2(y-1), ~(x-2)y(1-y)exp(y))$.  We set Lam$\acute{e}$'s parameter $\lambda=\mu=1$ for this problem.
\par
\noindent

\vspace{0.2 cm}

\begin{table}
\begin{tabular}{ |c|c|c| } 
 \hline
 ${\b{h}}$ & \textbf{error} & \textbf{order of conv. }\\
 \hline
 $2^{-1}$ &  3.2583 $\times$ $10^{-1}$~~ & -\\ 
 $2^{-2}$ & 8.8548$\times$ $10^{-2}$ ~~ & 1.8795\\ 
 $2^{-3}$ &  2.2846 $\times$ $10^{-2}$~~ &  1.9545 \\ 
 $2^{-4}$ &5.7886$\times$ $10^{-3}$ ~~ & 1.9806\\
 $2^{-5}$ & 1.4560$\times$ $10^{-3}$~~ & 1.9911 \\ 
 \hline
\end{tabular}
\quad
~~~
\begin{tabular}{ |c|c|c| } 
 \hline
 $\b{h}$ & \textbf{error} & \textbf{order of conv.}\\
 \hline
 $2^{-1}$ &  3.1989 $\times$ $10^{-1}$~~ & -\\ 
 $2^{-2}$ & 8.7572 $\times$ $10^{-2}$ ~~ & 1.8690\\ 
 $2^{-3}$ & 2.2658 $\times$ $10^{-2}$~~ &  1.9504 \\ 
 $2^{-4}$ & 5.7474$\times$ $10^{-3}$ ~~ & 1.9790\\
 $2^{-5}$ & 1.4463 $\times$ $10^{-3}$~~ & 1.9904 \\ 
 \hline
\end{tabular}
\bigbreak
\caption{Errors and orders of
convergence for SIPG and NIPG methods on uniform mesh for Model Problem 1.}
\label{table1}
\end{table}

\begin{figure}
	\begin{subfigure}[b]{0.45\textwidth}
		\includegraphics[width=\linewidth]
		{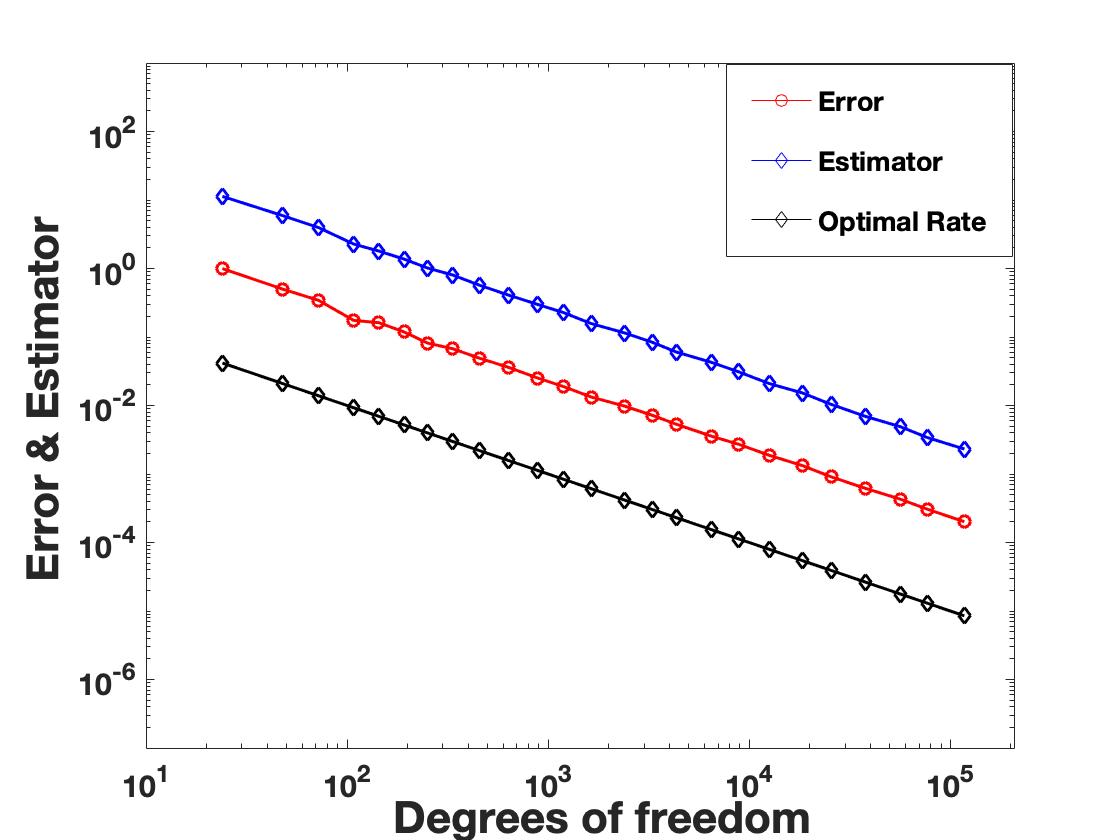}
		\caption{SIPG}
	\end{subfigure}
	\begin{subfigure}[b]{0.45\textwidth}
		\includegraphics[width=\linewidth]{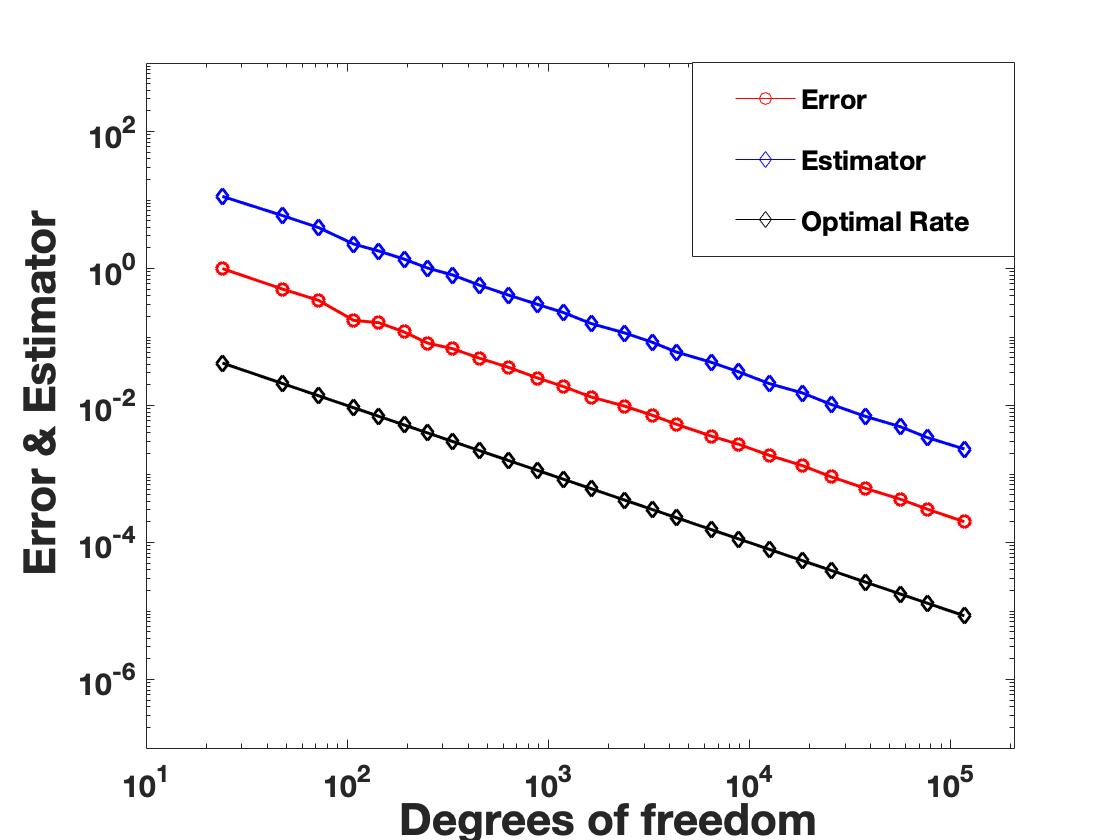}
		\caption{NIPG}
	\end{subfigure}
	\caption{Convergence of error and estimator for SIPG and NIPG method for Model Problem 1.}\label{Estmiator1}
\end{figure}

\begin{figure}
	\begin{subfigure}[b]{0.45\textwidth}
		\includegraphics[width=\linewidth]
		{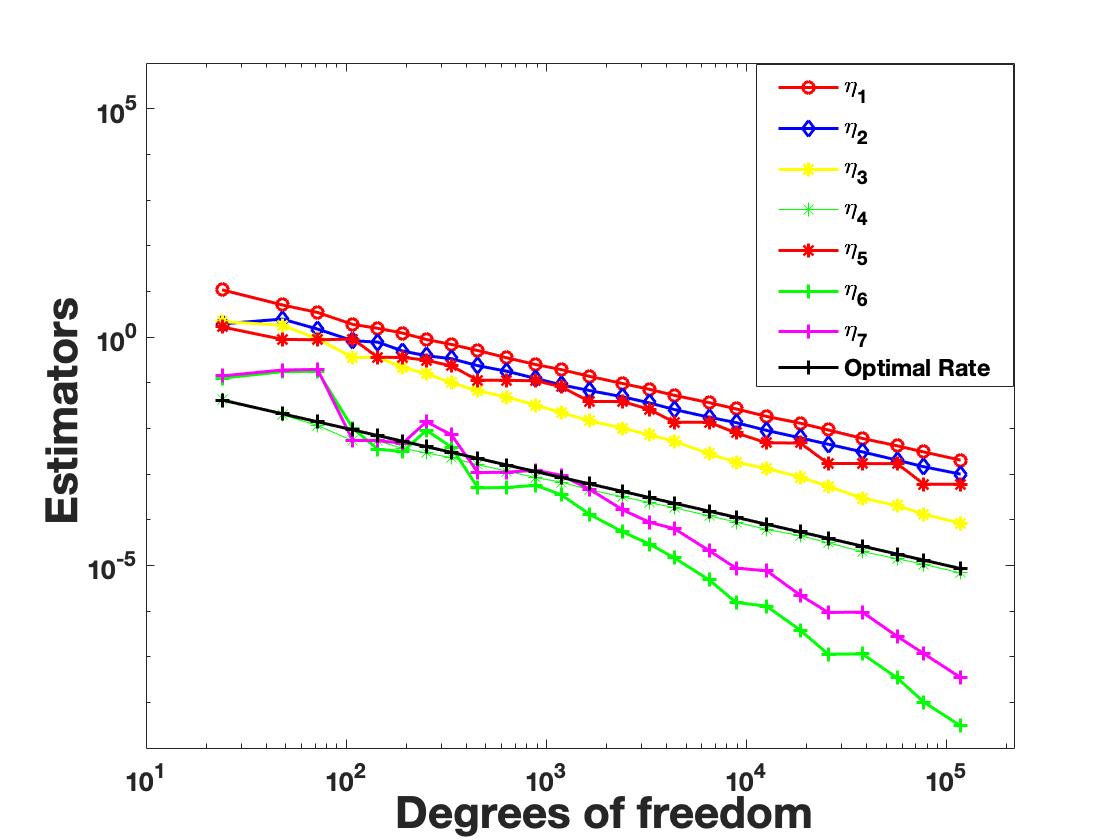}
		\caption{SIPG}
	\end{subfigure}
	\begin{subfigure}[b]{0.45\textwidth}
		\includegraphics[width=\linewidth]{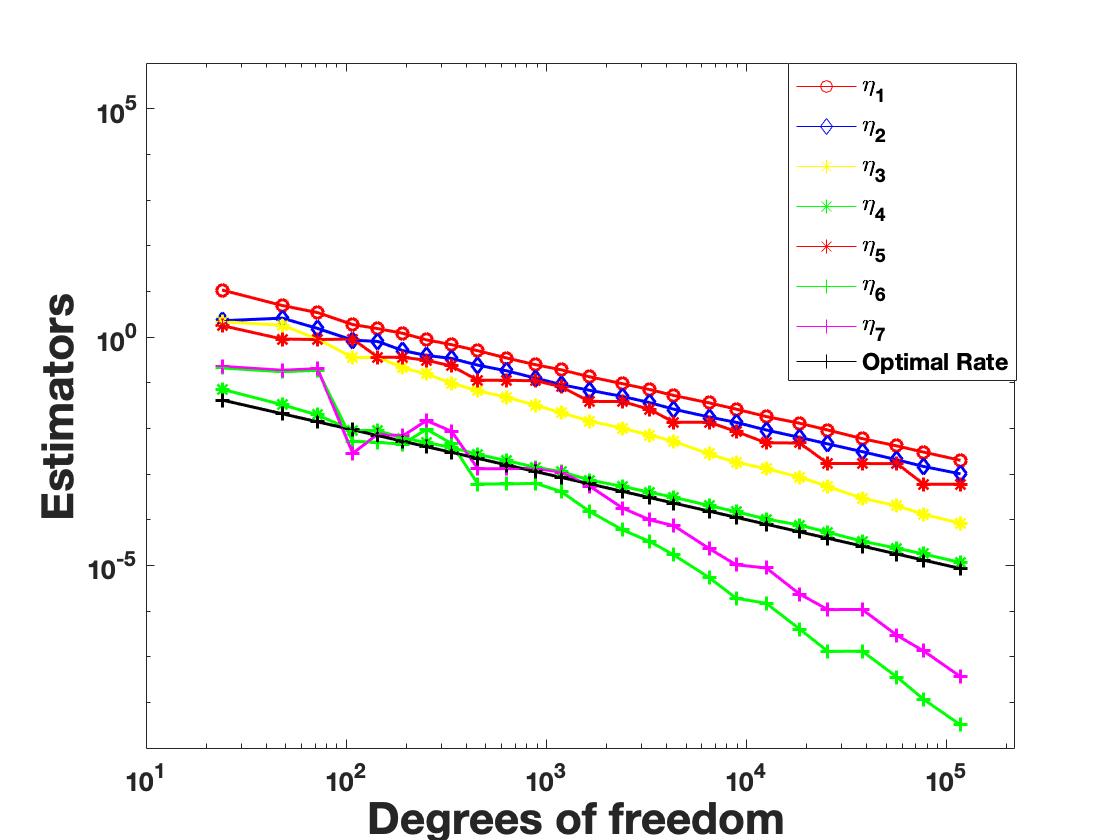}
		\caption{NIPG}
	\end{subfigure}
	\caption{Error estimator contributions for SIPG and NIPG method for Model Problem 1.}\label{Estimator11}
\end{figure} 

\begin{figure}
	\begin{subfigure}[b]{0.45\textwidth}
		\includegraphics[width=\linewidth]
		{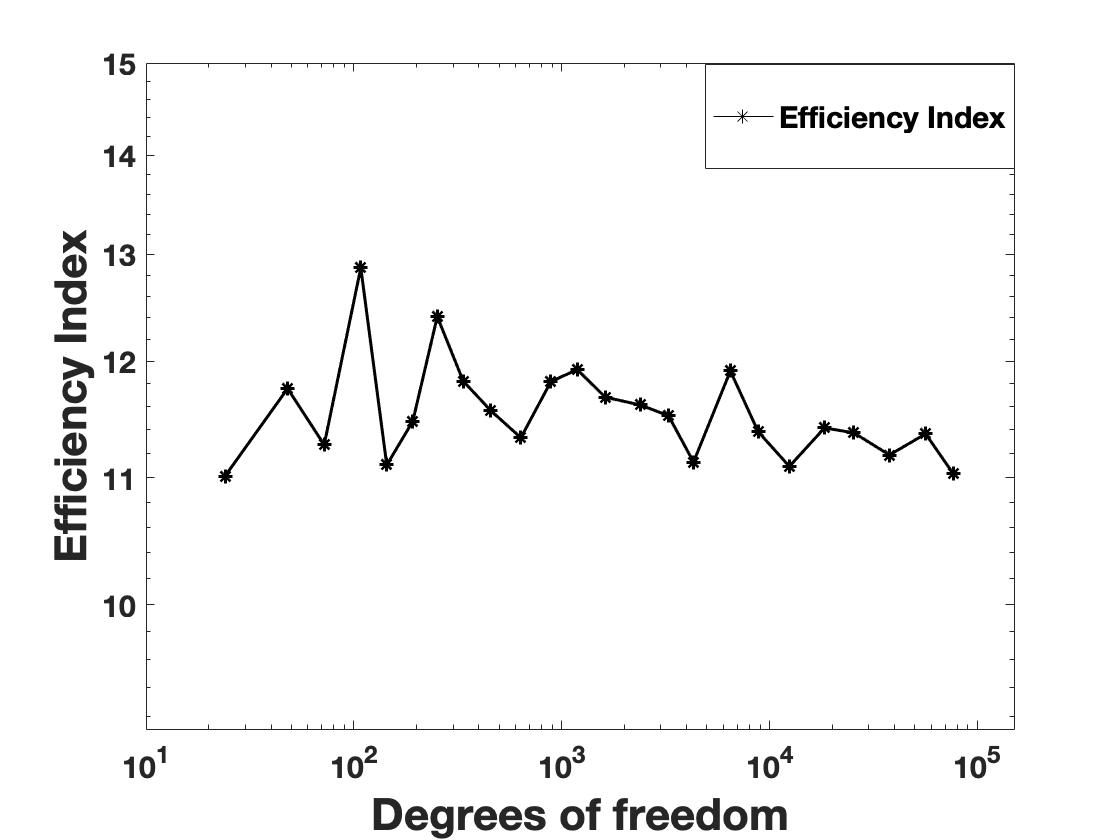}
		\caption{SIPG}
	\end{subfigure}
	\begin{subfigure}[b]{0.45\textwidth}
		\includegraphics[width=\linewidth]{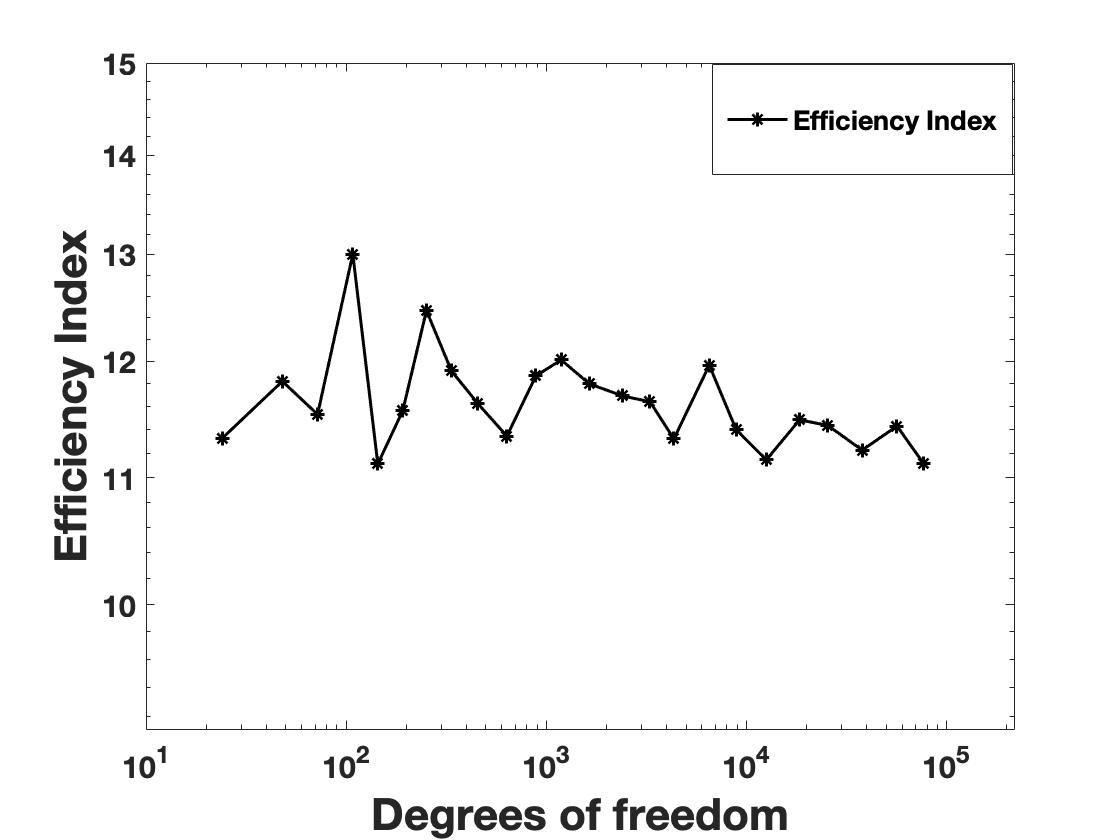}
		\caption{NIPG}
	\end{subfigure}
	\caption{Plot of efficiency index for SIPG and NIPG method for Model Problem 1.}\label{Efficiency Index}
\end{figure} 

\noindent
First,  we  conduct the numerical test on uniform refined meshes for SIPG and NIPG methods.  The smoothness of the exact solution ensures the optimal order of convergence(1/NDof,   NDof= Degrees of freedom). Table \ref{table1} demonstrates the convergence history of energy norm error for SIPG and NIPG methods,  respectively on uniform refinement.  Next,  we  conduct the test on adaptive mesh which operates in the following way:

\begin{itemize}
\item \textbf{SOLVE}: In this step we compute the solution $\b{u_h}$ of discrete variational inequality \eqref{discrete}.
\item \textbf{ESTIMATE}: The error estimator $\eta_h$  is computed elementwise in this step.
\item \textbf{MARK}: We use D\"orlfer marking strategy \cite{Dorfler:1996:Afem} with the parameter $\theta =0.4$ to mark the triangles.
\item \textbf{REFINE}: Finally, in this step,  we refine the marked triangles using newest vertex bisection algorithm \cite{Verfurth:1995:AdaptiveBook} and obtain a finer mesh.
\end{itemize}
\vspace{0.5 cm}
In Figure \ref{Estmiator1},  we plot the DG norm error and residual estimator $\eta_h$ on the adaptive mesh.  Therein,  we observe that the error and estimator are converging with the optimal rate with the increase in degrees of freedom,   thus ensuring the reliability of the error estimators.  Figure \ref{Estimator11} illustrates the convergence of individual error estimators on the adaptive mesh.  The efficiency index which is calculated as ratio of estimator and error is depicted in Figure \ref{Efficiency Index}.  We can see that the graph of efficiency index is both bounded above and below by generic constant which ensures that the residual estimator $\eta_h$ is both reliable and efficient.
\vspace{0.5 cm}
\par
\noindent
\begin{figure}
	\begin{subfigure}[b]{0.45\textwidth}
		\includegraphics[width=\linewidth]
		{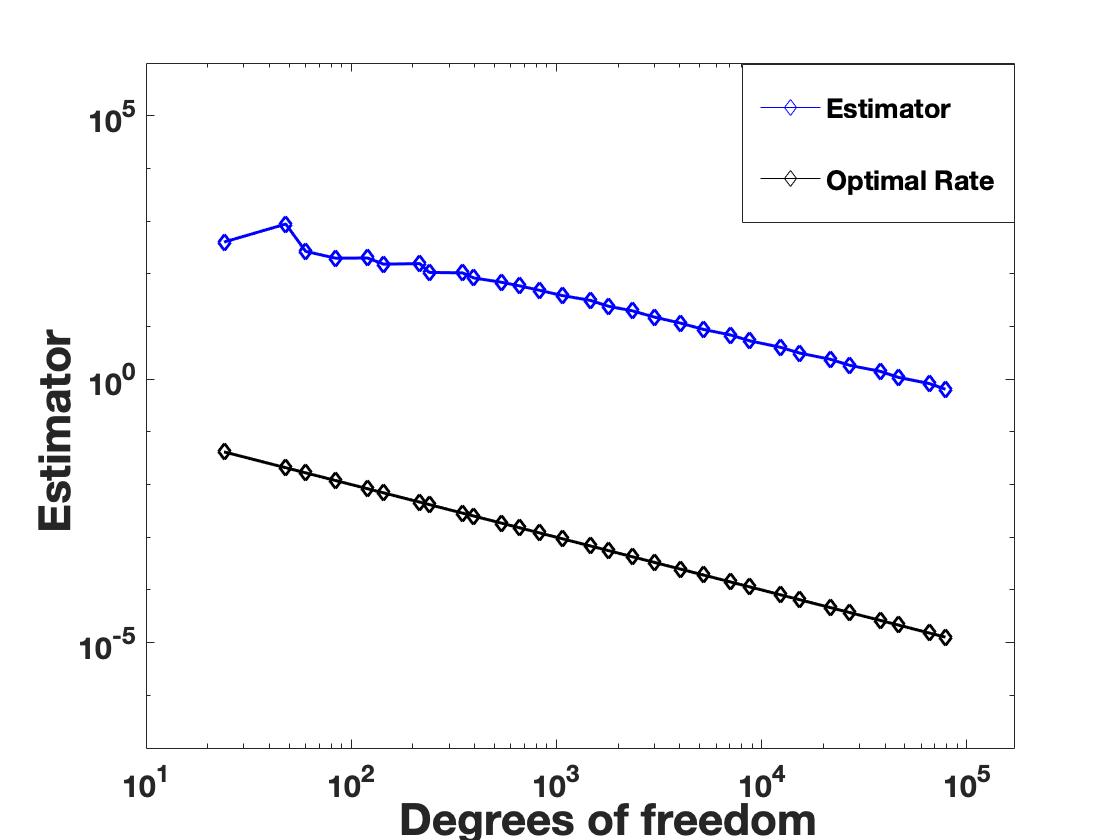}
		\caption{SIPG}
	\end{subfigure}
	\begin{subfigure}[b]{0.45\textwidth}
		\includegraphics[width=\linewidth]{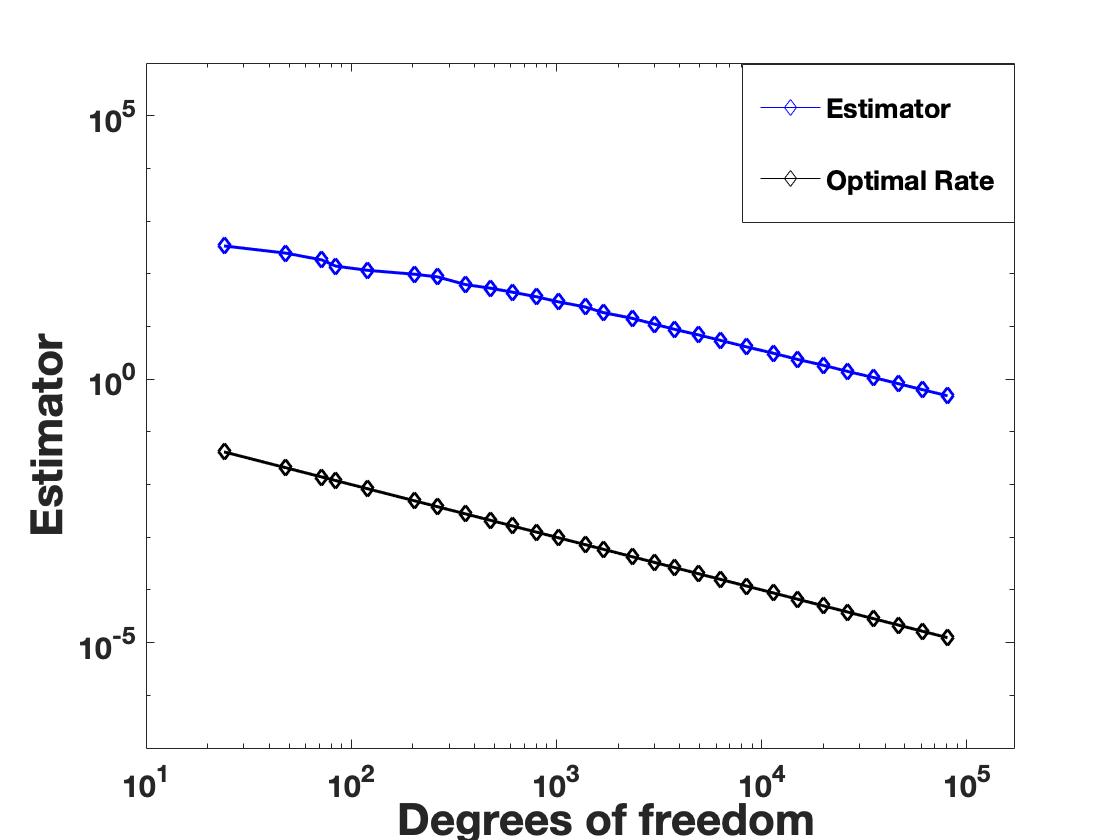}
		\caption{NIPG}
	\end{subfigure}
	\caption{Convergence of estimator for SIPG and NIPG method for Model Problem 2.}\label{Adpative mesh1}
\end{figure} 
\begin{figure}
	\begin{subfigure}[b]{0.45\textwidth}
		\includegraphics[width=\linewidth]
		{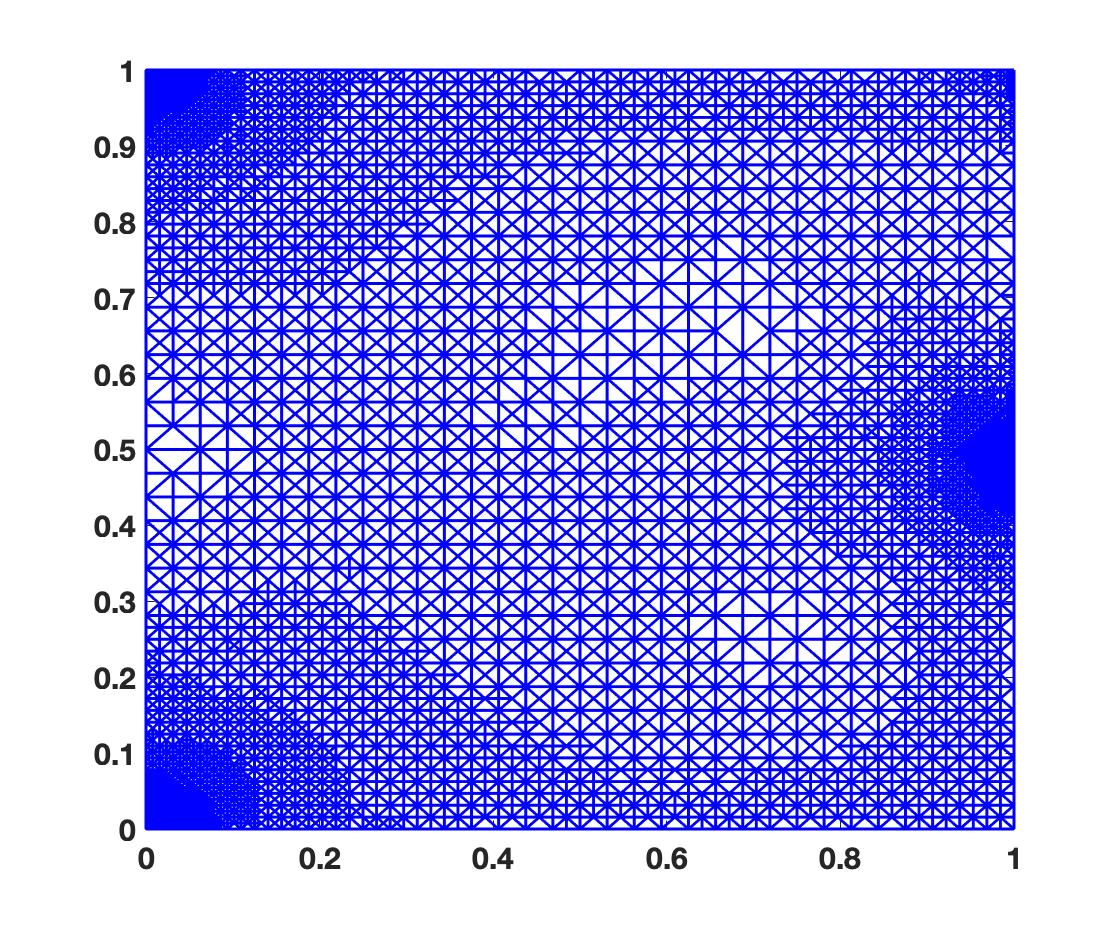}
		\caption{SIPG}
	\end{subfigure}
	\begin{subfigure}[b]{0.45\textwidth}
		\includegraphics[width=\linewidth]{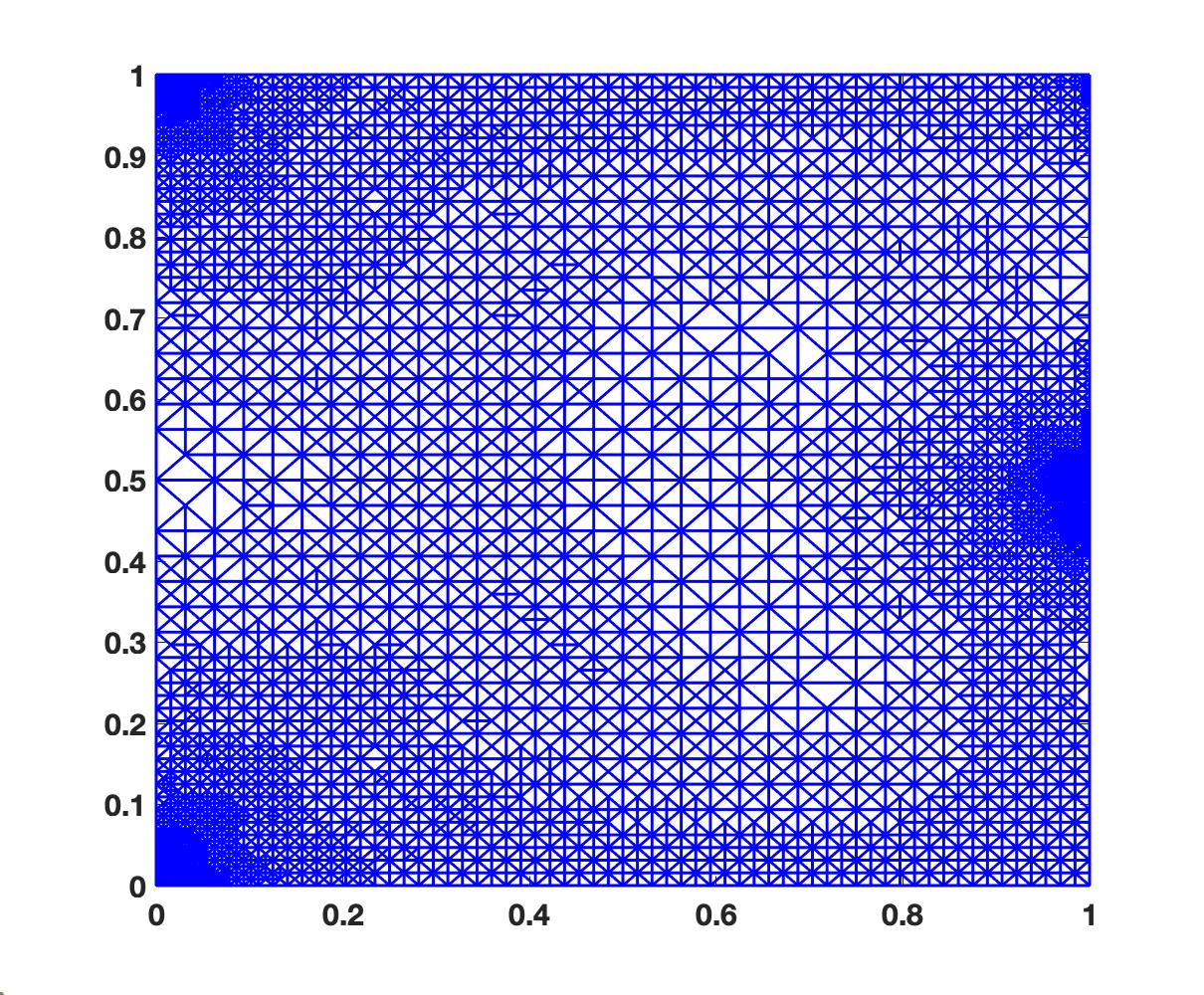}
		\caption{NIPG}
	\end{subfigure}
	\caption{Adaptive mesh refinement for SIPG and NIPG method for Model Problem 2.}\label{Adpative mesh}
\end{figure} 

\noindent
\underline{\textbf{Model Problem 2:}} \textit{(Contact with a rigid wedge)}
\vspace{0.2 cm}
\par
\noindent
In the next model problem,  the domain $\Omega =(0,1)\times(0,1)$ is the cross section of the linear elastic square which is  pushed towards a non-zero obstacle $h(y) = -0.2+\rvert 0.5-y \lvert $ in the direction of potential contact boundary $ \{1\} \times (0,1)$ .  The Neumann boundary $\big((0,1) \times \{0\} \big) \cup \big((0,1) \times\{1\}\big)$ is traction free i.e.  $\b{g}~=~0$.  Next,  we impose non-homogeneous Dirichlet boundary condition $\b{u}=(-0.1, 0 )$ on $ \{0\} \times (0,1)$. The load vector $\b{f}$ is assumed to be zero.  The Lam$\acute{e}$ parameters $\mu$ and $\lambda$ are computed by
\begin{align*}
\mu = \frac{E}{2(1+\nu)}~~ \text{and}~~\lambda = \frac{E\nu}{(1+\nu)(1-2\nu)},
\end{align*}
where,  the Young’s modulus $E$ and the Poisson ratio $\nu$ are 500 and 0.3,  respectively for the given model problem. 
\vspace{0.3 cm}
\par

In contrast to Model Problem 1,  we do not have exact solution for this problem,  therefore we only plot the convergence of residual error estimator versus the degrees of freedom.  The behaviour of residual error estimator for SIPG and NIPG methods is described in Figure \ref{Adpative mesh1}.  Therein,  the figure depicts the optimal order of convergence of the estimators.  Figure \ref{Adpative mesh}  illustrates the adaptively refined grid steered by the residual error estimator for both the DG methods.  As expected,  the mesh is strongly refined near the  area where the tip of wedge is coming in the contact with the unit square,  free boundary around the contact zone and the intersection of the Neumann and Dirichlet boundaries.  Thus,  the deformation of the body under the effect of traction is well captured.  The decay of individual error estimator $\eta_i,  1 \leq i \leq 7$ is plotted in Figure \ref{IndEstR1}.  It is evident that each estimator contribution is converging with the optimal rate.

\begin{figure}
	\begin{subfigure}[b]{0.45\textwidth}
		\includegraphics[width=\linewidth]
		{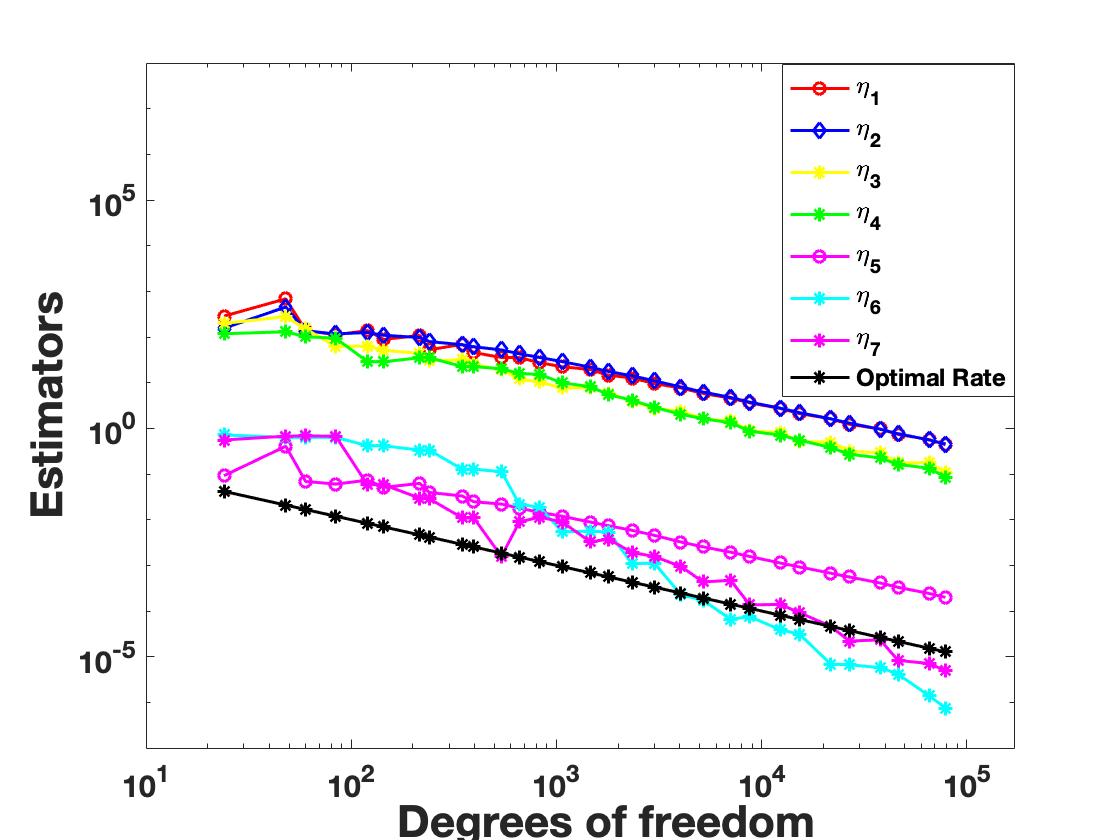}
		\caption{SIPG}
	\end{subfigure}
	\begin{subfigure}[b]{0.45\textwidth}
		\includegraphics[width=\linewidth]{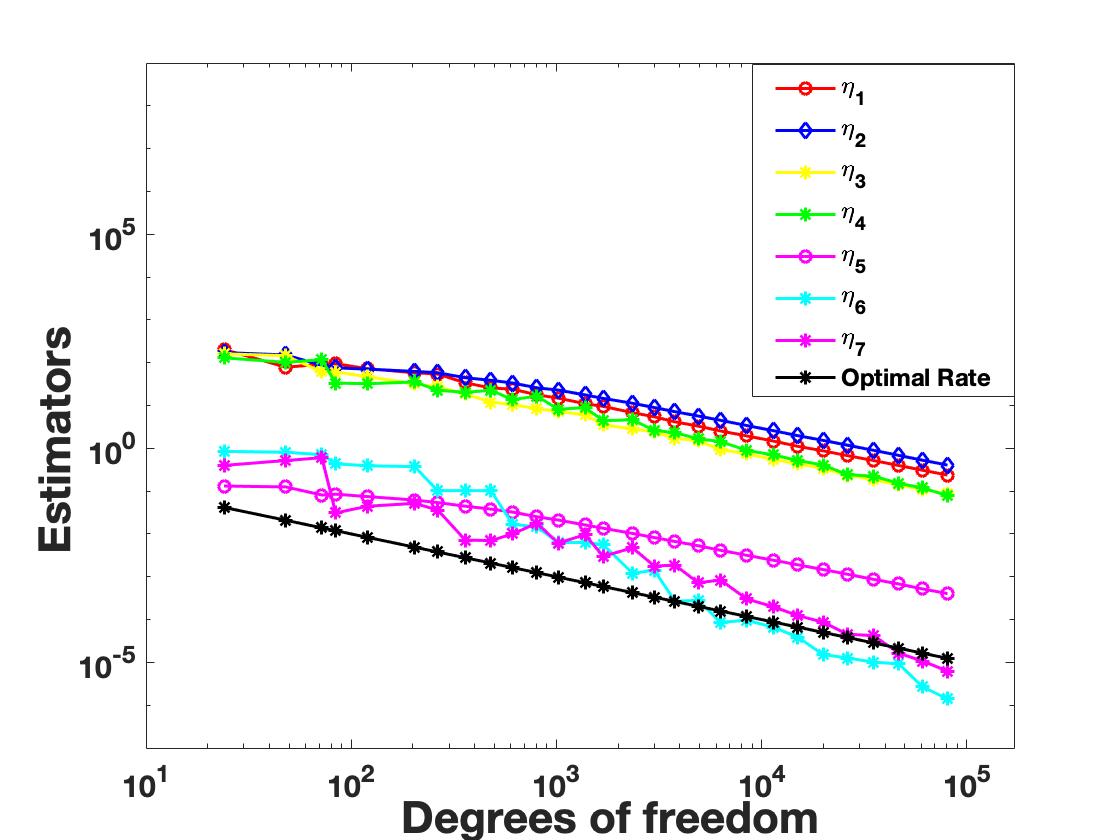}
		\caption{NIPG}
	\end{subfigure}
	\caption{Residual error estimators contribution for SIPG and NIPG method for Model Problem 2.}\label{IndEstR1}
\end{figure} 

\end{document}